\documentclass[12pt,english]{article}
\pdfoutput=1
\usepackage[T1]{fontenc}
\usepackage[latin9]{inputenc}
\usepackage{geometry}
\usepackage{array}
\usepackage{float}
\usepackage{mathtools}
\usepackage{amsmath}
\usepackage{amsthm}
\usepackage{amssymb}

\makeatletter

\providecommand{\tabularnewline}{\\}

\numberwithin{equation}{section}
\numberwithin{figure}{section}
\theoremstyle{plain}
\newtheorem{thm}{\protect\theoremname}[section]
\theoremstyle{plain}
\newtheorem{lem}[thm]{\protect\lemmaname}
\theoremstyle{plain}
\newtheorem{prop}[thm]{\protect\propositionname}
\theoremstyle{plain}
\newtheorem{conjecture}[thm]{\protect\conjecturename}

\usepackage{bm}
\usepackage{hyperref}
\usepackage{colortbl}
\usepackage{xparse}
\usepackage{caption}
\usepackage{microtype}

\DeclareMathOperator{\Des}{Des}
\DeclareMathOperator{\pk}{pk}
\DeclareMathOperator{\ipk}{ipk}
\DeclareMathOperator{\lpk}{lpk}
\DeclareMathOperator{\ilpk}{ilpk}
\DeclareMathOperator{\Comp}{Comp}
\DeclareMathOperator{\udComp}{udComp}
\DeclareMathOperator{\br}{br}
\DeclareMathOperator{\ibr}{ibr}
\DeclareMathOperator{\udr}{udr}
\DeclareMathOperator{\iudr}{iudr}
\DeclareMathOperator{\des}{des}
\DeclareMathOperator{\ides}{ides}
\DeclareMathOperator{\st}{st}
\DeclareMathOperator{\ist}{ist}
\DeclareMathOperator{\maj}{maj}
\DeclareMathOperator{\imaj}{imaj}
\DeclareMathOperator{\Pin}{Pin}

\usepackage{titlesec}
\titleformat{\section}
  {\normalfont\large\bfseries}{\thesection.}{.8ex}{} %changed from Large
\titleformat{\subsection}
  {\normalfont\bfseries}{\thesubsection.}{.8ex}{} %large removed
\let\originalleft\left
\let\originalright\right
\renewcommand{\left}{\mathopen{}\mathclose\bgroup\originalleft}
\renewcommand{\right}{\aftergroup\egroup\originalright}

\newcommand{\leqnomode}{\tagsleft@true\let\veqno\@@leqno}
\newcommand{\reqnomode}{\tagsleft@false\let\veqno\@@eqno}
\reqnomode

\usepackage{authblk}
\title{Two-sided permutation statistics via \\ symmetric functions}
\author[1]{Ira M.\ Gessel\thanks{\tt{gessel@brandeis.edu}}} 
\author[2]{Yan Zhuang\thanks{\tt{yazhuang@davidson.edu}}}
\affil[1]{Department of Mathematics, Brandeis University} 
\affil[2]{Department of Mathematics and Computer Science, Davidson College}
\date{\today}

\makeatother

\usepackage{babel}
\providecommand{\conjecturename}{Conjecture}
\providecommand{\lemmaname}{Lemma}
\providecommand{\propositionname}{Proposition}
\providecommand{\theoremname}{Theorem}

\begin{document}

\maketitle

\begin{abstract}
Given a permutation statistic $\st$, define its inverse statistic
$\ist$ by $\ist(\pi)\coloneqq\st(\pi^{-1})$. We give a general approach,
based on the theory of symmetric functions, for finding the joint
distribution of $\st_{1}$ and $\ist_{2}$ whenever $\st_{1}$ and
$\st_{2}$ are descent statistics: permutation statistics
that depend only on the descent composition. We apply this method
to a number of descent statistics, including the descent number, the
peak number, the left peak number, the number
of up-down runs, and the major index. Perhaps surprisingly, in many cases
the polynomial giving the joint distribution of $\st_{1}$ and $\ist_{2}$
can be expressed as a simple sum involving products of the polynomials
giving the (individual) distributions of $\st_{1}$ and $\st_{2}$.
Our work leads to a rederivation of Stanley's generating function
for doubly alternating permutations, as well as several conjectures 
concerning real-rootedness and $\gamma$-positivity.
\end{abstract}
\textbf{\small{}Keywords:}{\small{} permutation statistics, descents,
peaks, two-sided Eulerian polynomials, doubly alternating permutations, symmetric functions}{\let\thefootnote\relax\footnotetext{YZ was partially supported by an AMS-Simons Travel Grant and NSF grant DMS-2316181.}}{\let\thefootnote\relax}

\tableofcontents{}

\section{Introduction}

Let $\mathfrak{S}_{n}$ denote the symmetric group of permutations
of the set $[n]\coloneqq\{1,2,\dots,n\}$. We call $i\in[n-1]$ a
\textit{descent} of $\pi\in\mathfrak{S}_{n}$ if $\pi(i)>\pi(i+1)$.
Let
\[
\des(\pi)\coloneqq\sum_{\pi(i)>\pi(i+1)}1\quad\text{and}\quad\maj(\pi)\coloneqq\sum_{\pi(i)>\pi(i+1)}i
\]
be the number of descents and the sum of descents of $\pi$, respectively.
The \textit{descent number} $\des$ and \textit{major index} $\maj$
are classical permutation statistics dating back to MacMahon \cite{macmahon}.

Given a permutation statistic $\st$, let us define its \textit{inverse
statistic} $\ist$ by $\ist(\pi)\coloneqq\st(\pi^{-1})$. This paper
is concerned with the general problem of finding the joint distribution
of a permutation statistic $\st$ and its inverse statistic $\ist$
over the symmetric group $\mathfrak{S}_{n}$. This was first done
by Carlitz, Roselle, and Scoville \cite{Carlitz1966} for $\des$
and by Roselle \cite{Roselle1974} for $\maj$. Among the results
of Carlitz, Roselle, and Scoville was the elegant generating function
formula
\begin{equation}
\sum_{n=0}^{\infty}\frac{A_{n}(s,t)}{(1-s)^{n+1}(1-t)^{n+1}}x^{n}=\sum_{i,j=0}^{\infty}\frac{s^{i}t^{j}}{(1-x)^{ij}}\label{e-2seur}
\end{equation}
for the \textit{two-sided Eulerian polynomials} $A_{n}(s,t)$ defined
by
\[
A_{n}(s,t)\coloneqq\sum_{\pi\in\mathfrak{S}_{n}}s^{\des(\pi)+1}t^{\ides(\pi)+1}
\]
for $n\geq1$ and $A_{0}(s,t)\coloneqq1$.\footnote{Carlitz, Roselle, and Scoville actually considered the joint distribution
of ascents (elements of $[n-1]$ which are not descents) and inverse
ascents, but this is the same as the joint distribution of descents and
inverse descents by symmetry.} (Throughout this paper, all polynomials encoding distributions of
permutation statistics will be defined to be 1 for $n=0$.) Roselle gave a similar formula for the joint distribution of $\maj$ and $\imaj$ over $\mathfrak{S}_{n}$. Note that extracting coefficients of $x^{n}$
from both sides of (\ref{e-2seur}) yields the formula
\begin{equation}
\label{e-crs}
\frac{A_{n}(s,t)}{(1-s)^{n+1}(1-t)^{n+1}}=\sum_{i,j=0}^{\infty}{ij+n-1 \choose n}s^{i}t^{j},
\end{equation}
which Petersen \cite{Petersen2013} later proved using the technology
of putting balls in boxes. Equation \eqref{e-crs} may be compared with the formula 
\begin{equation}
\label{e-Eul}
\frac{A_{n}(t)}{(1-t)^{n+1}}=\sum_{k=0}^{\infty}k^{n}t^{k}
\end{equation}
for the (ordinary) \emph{Eulerian polynomials}
$$
A_n(t) \coloneqq \sum_{\pi\in \mathfrak{S}_n}t^{\des(\pi)+1}.
$$

Several years after the work of Carlitz--Roselle--Scoville and Roselle, Garsia and Gessel~\cite{Garsia1979} used the theory of $P$-partitions to derive the formula 
\begin{equation}
\sum_{n=0}^{\infty}\frac{A_{n}(s,t,q,r)}{(1-s)(1-qs)\cdots(1-q^{n}s)(1-t)(1-rt)\cdots(1-r^{n}t)}x^{n}=\sum_{i,j=0}^{\infty}s^{i}t^{j}\prod_{k=0}^{i}\prod_{l=0}^{j}\frac{1}{1-xq^{k}r^{l}}\label{e-gg}
\end{equation}
for the quadrivariate polynomials
\[
A_{n}(s,t,q,r)\coloneqq\sum_{\pi\in\mathfrak{S}_{n}}s^{\des(\pi)}t^{\ides(\pi)}q^{\maj(\pi)}r^{\imaj(\pi)}.
\]
The Garsia\textendash Gessel formula (\ref{e-gg}) specializes to
both the Carlitz\textendash Roselle\textendash Scoville formula (\ref{e-2seur})
for $(\des,\ides)$ as well as Roselle's formula for $(\maj,\imaj)$.

The two-sided Eulerian polynomials have since received renewed attention due to a refined $\gamma$-positivity conjecture of Gessel, which was later proven by Lin \cite{Lin2016a}. The joint statistic $(\des,\maj,\ides,\imaj)$ and the sum $\des+\ides$ have also been studied in the probability literature; see, for example, \cite{Brueck2022, Chatterjee2017, Feray2020, Vatutin1996}.

Throughout this paper, let us call a pair of the form $(\st,\ist)$
a \textit{two-sided permutation statistic} and the distribution of
this statistic over $\mathfrak{S}_{n}$ the \textit{two-sided distribution}
of $\st$. If we are taking $\st$ to be a pair $(\st_{1},\st_{2})$
such as $(\des,\maj)$, then we consider $(\st_{1},\ist_{1},\st_{2},\ist_{2})$
a two-sided statistic as well.

\subsection{Descent statistics}

We use the notation $L\vDash n$ to indicate that $L$ is a composition
of $n$. Every permutation can be uniquely decomposed into a sequence
of maximal increasing consecutive subsequences\textemdash or equivalently,
maximal consecutive subsequences with no descents\textemdash which
we call \textit{increasing runs}. The descent composition of $\pi$,
denoted $\Comp(\pi)$, is the composition whose parts are the lengths
of the increasing runs of $\pi$ in the order that they appear. For
example, the increasing runs of $\pi=72485316$ are $7$, $248$,
$5$, $3$, and $16$, so $\Comp(\pi)=(1,3,1,1,2)$. 

A permutation statistic $\st$ is called a \textit{descent statistic}
if $\Comp(\pi)=\Comp(\sigma)$ implies $\st(\pi)=\st(\sigma)$\textemdash that
is, if $\st$ depends only on the descent composition. Whenever $\st$
is a descent statistic, we may write $\st(L)$ for the value of $\st$
on any permutation with descent composition $L$. Both $\des$ and
$\maj$ are descent statistics, and in this paper, we will also consider
the following descent statistics:
\begin{itemize}
\item The \textit{peak number} $\pk$. We call $i$ (where $2\leq i\leq n-1$)
a \textit{peak} of $\pi\in\mathfrak{S}_{n}$ if $\pi(i-1)<\pi(i)>\pi(i+1)$.
Then $\pk(\pi)$ is the number of peaks of $\pi$.
\item The \textit{left peak number} $\lpk$. We call $i\in[n-1]$ a \textit{left
peak} of $\pi\in\mathfrak{S}_{n}$ if either $i$ is a peak of $\pi$,
or if $i=1$ and $\pi(1)>\pi(2)$. Then $\lpk(\pi)$ is the number
of left peaks of $\pi$.
\item The \textit{number of biruns} $\br$. A \textit{birun} of a permutation
$\pi$ is a maximal monotone consecutive subsequence. Then $\br(\pi)$
is the number of biruns of $\pi$.
\item The \textit{number of up-down runs} $\udr$. An \textit{up-down run}
of $\pi$ is either a birun of $\pi$, or $\pi(1)$ if $\pi(1)>\pi(2)$.
Then $\udr(\pi)$ is the number of up-down runs of $\pi$.\footnote{Equivalently, $\udr(\pi)$ is equal to the length of the longest alternating
subsequence of $\pi$, which was studied in depth by Stanley \cite{Stanley2008}.}
\end{itemize}
For example, if $\pi=624731598$, then the peaks of $\pi$ are $4$
and $8$; the left peaks of $\pi$ are $1$, $4$, and $8$; the biruns
of $\pi$ are $62$, $247$, $731$, $159$, and $98$; and the up-down
runs of $\pi$ are $6$, $62$, $247$, $731$, $159$, and $98$.
Therefore, we have $\pk(\pi)=2$, $\lpk(\pi)=3$, $\br(\pi)=5$, and
$\udr(\pi)=6$. Other examples of descent statistics include the number
of valleys, double ascents, double descents, and alternating descents
(see \cite{Zhuang2016} for definitions).

In this paper, we give a general approach to the problem of finding
the two-sided distribution of $\st$ whenever $\st$ is a descent
statistic. In fact, our approach can be used to find distributions
of \textit{mixed two-sided statistics}: the joint distribution of
$\st_{1}$ and $\ist_{2}$ when $\st_{1}$ and $\st_{2}$ are (possibly
different) descent statistics. Our approach utilizes a theorem of
Foulkes \cite{Foulkes1976} at the intersection of permutation enumeration
and symmetric function theory: the number of permutations $\pi$ with
prescribed descent composition $L$ whose inverse $\pi^{-1}$ has
descent composition $M$ is equal to the scalar product $\left\langle r_{L},r_{M}\right\rangle $
of two ribbon skew Schur functions (defined in Section \ref{ss-foulkes}).
Thus, if $f$ is a generating function for ribbon functions that keeps
track of a descent statistic $\st_{1}$ and $g$ is a similar generating
function for a descent statistic $\st_{2}$, then the scalar product
$\left\langle f,g\right\rangle $ should give a generating function
for $(\st_{1},\ist_{2})$.

To illustrate this idea, let us sketch how our approach can be used
to rederive the formula (\ref{e-2seur}) of Carlitz, Roselle, and
Scoville for the two-sided Eulerian polynomials. We have 
\[
\frac{1}{1-tH(x)}=\frac{1}{1-t}+\sum_{n=1}^{\infty}\sum_{L\vDash n}\frac{t^{\des(L)+1}}{(1-t)^{n+1}}r_{L}x^{n}
\]
where $H(z)\coloneqq\sum_{n=0}^{\infty}h_{n}z^{n}$ is the ordinary
generating function for the complete symmetric functions $h_{n}$.
Then, upon applying Foulkes's theorem, we get
\begin{align*}
\left\langle \frac{1}{1-sH(x)},\frac{1}{1-tH(1)}\right\rangle  & =\frac{1}{(1-s)(1-t)}+\sum_{n=1}^{\infty}\sum_{L,M\vDash n}\frac{s^{\des(L)+1}t^{\des(M)+1}}{(1-s)^{n+1}(1-t)^{n+1}}\left\langle r_{L},r_{M}\right\rangle x^{n}\\
 & =\sum_{n=0}^{\infty}\frac{A_{n}(s,t)}{(1-s)^{n+1}(1-t)^{n+1}}x^{n}.
\end{align*}
However, calculating the same scalar product but in a different way
yields 
\[
\left\langle \frac{1}{1-sH(x)},\frac{1}{1-tH(1)}\right\rangle =\sum_{i,j=0}^{\infty}\frac{s^{i}t^{j}}{(1-x)^{ij}},
\]
and equating these two expressions recovers (\ref{e-2seur}).

In fact, a third way of calculating the same scalar product leads
to a new formula expressing the two-sided Eulerian polynomials $A_{n}(s,t)$
in terms of the Eulerian polynomials $A_n(t)$.
Let us use the notations $\lambda\vdash n$ and $\left|\lambda\right|=n$
to indicate that $\lambda$ is a partition of $n$, and let $l(\lambda)$
denote the number of parts of $\lambda$. We write $\lambda=(1^{m_{1}}2^{m_{2}}\cdots)$
to mean that $\lambda$ has $m_{1}$ parts of size $1$, $m_{2}$
parts of size 2, and so on, and define $z_{\lambda}\coloneqq1^{m_{1}}m_{1}!\,2^{m_{2}}m_{2}!\,\cdots$.
Then it can be shown that 
\[
\left\langle \frac{1}{1-sH(x)},\frac{1}{1-tH(1)}\right\rangle =\sum_{\lambda}\frac{1}{z_{\lambda}}\frac{A_{l(\lambda)}(s)A_{l(\lambda)}(t)}{(1-s)^{l(\lambda)+1}(1-t)^{l(\lambda)+1}}x^{\left|\lambda\right|},
\]
leading to
\[
A_{n}(s,t)=\sum_{\lambda\vdash n}\frac{1}{z_{\lambda}}((1-s)(1-t))^{n-l(\lambda)}A_{l(\lambda)}(s)A_{l(\lambda)}(t).
\]
Collecting terms with $l(\lambda)=k$ then gives the formula
\begin{equation*}
A_n(s,t) = \frac{1}{n!} \sum_{k=0}^n c(n,k) ((1-s)(1-t))^{n-k}A_k(s) A_k(t),
\end{equation*}
(see Theorem \ref{t-des}), where the $c(n,k)$ are the unsigned Stirling numbers of the first kind \cite[p.~26]{Stanley2011}. Perhaps surprisingly, many of the ``two-sided polynomials'' that we study in this paper can be similarly expressed
as a simple sum involving products of the univariate polynomials encoding
the distributions of the individual statistics.

Foulkes's theorem was also used by Stanley \cite{Stanley2007} in his study of alternating permutations; some of our results, notably Theorems \ref{t-pkipkgf} and \ref{t-lpkilpkgf}, generalize results of his.

\subsection{Outline}

We organize this paper as follows. Section \ref{s-symfun} is devoted
to background material on symmetric functions. While we assume familiarity
with basic symmetric function theory at the level of Stanley \cite[Chapter 7]{Stanley2001},
we shall use this section to establish notation, recall some elementary
facts that will be important for our work, and to give an exposition
of various topics and results needed to develop our approach to two-sided
statistics; these include Foulkes's theorem, plethysm, and symmetric
function generating functions associated with descent statistics.

Our main results are given in Sections \ref{s-pkdes}\textendash \ref{s-maj}.
We begin in Section \ref{s-pkdes} by using our symmetric function
approach to prove formulas for the two-sided statistic $(\pk,\ipk,\des,\ides)$,
which we then specialize to formulas for $(\pk,\ipk)$, $(\pk,\ides)$,
and $(\des,\ides)$. We continue in Section \ref{s-lpk} by deriving
analogous formulas for $(\lpk,\ilpk,\des,\ides)$ and $(\lpk,\ipk,\des,\ides)$,
and their specializations. Notably, our results for the two-sided distributions of $\pk$ 
and of $\lpk$ lead to a rederivation of Stanley's \cite{Stanley2007} generating function
formula for ``doubly alternating permutations'': alternating permutations whose inverses 
are alternating.

Section \ref{s-udr} considers (mixed) two-sided distributions involving 
the number of up-down runs, as well as a couple involving the number of biruns. 
In Section \ref{s-maj}, we give a rederivation of the Garsia\textendash Gessel 
formula for $(\maj,\imaj,\des,\ides)$ using our approach and give formulas for
several mixed two-sided distributions involving the major index.

We conclude in Section \ref{s-conj} with several conjectures concerning
real-rootedness and $\gamma$-positivity of some of the polynomials
appearing in our work.

\section{\label{s-symfun}Symmetric functions background}

Let $\Lambda$ denote the $\mathbb{Q}$-algebra of symmetric functions
in the variables $x_{1},x_{2},\dots$. We recall the important bases
for $\Lambda$: the monomial symmetric functions $m_{\lambda}$, the
complete symmetric functions $h_{\lambda}$, the elementary symmetric
functions $e_{\lambda}$, the power sum symmetric functions $p_{\lambda}$,
and the Schur functions $s_{\lambda}$. As usual, we write $h_{(n)}$
as $h_{n}$, $e_{(n)}$ as $e_{n}$, and $p_{(n)}$ as $p_{n}$.

We will also work with symmetric functions with coefficients involving
additional variables such as $s$, $t$, $y$, $z$, and $\alpha$,
as well as symmetric functions of unbounded degree like
\[
H(z)=\sum_{n=0}^{\infty}h_{n}z^{n}\quad\text{and}\quad E(z)\coloneqq\sum_{n=0}^{\infty}e_{n}z^{n}.
\]
We adopt the notation 
\[
H\coloneqq H(1)=\sum_{n=0}^{\infty}h_{n}\quad\text{and}\quad E\coloneqq E(1)=\sum_{n=0}^{\infty}e_{n}.
\]

\subsection{\label{ss-foulkes}The scalar product and Foulkes's theorem}

Let $\left\langle \cdot,\cdot\right\rangle \colon\Lambda\times\Lambda\rightarrow\mathbb{Q}$
denote the usual scalar product on symmetric functions defined by
\[
\left\langle m_{\lambda},h_{\mu}\right\rangle =\begin{cases}
1, & \text{if }\lambda=\mu\\
0, & \text{otherwise,}
\end{cases}
\]
for all partitions $\lambda$ and $\mu$ and extending bilinearly,
i.e., by requiring that $\{m_{\lambda}\}$ and $\{h_{\mu}\}$ be dual
bases. Then we have
\[
\left\langle p_{\lambda},p_{\mu}\right\rangle =\begin{cases}
z_{\lambda}, & \text{if }\lambda=\mu\\
0, & \text{otherwise,}
\end{cases}
\]
for all $\lambda$ and $\mu$ \cite[Proposition 7.9.3]{Stanley2011}.
We extend the scalar product in the obvious way to symmetric functions
involving other variables as well as symmetric functions of unbounded
degree. Note that the scalar product is not always defined for the
latter; e.g., we have $\left\langle H(z),H\right\rangle =\sum_{n=0}^{\infty}z^{n}$
but $\left\langle H,H\right\rangle $ is undefined.

Given a composition $L$, let $r_{L}$ denote the skew Schur function
of ribbon shape $L$. That is, for $L=(L_{1},L_{2},\dots,L_{k})$,
we have 
\[
r_{L}=\sum_{i_{1},\dots,i_{n}}x_{i_{1}}x_{i_{2}}\cdots x_{i_{n}}
\]
where the sum is over all $i_{1},\dots,i_{n}$ satisfying 
\[
\underset{L_{1}}{\underbrace{i_{1}\leq\cdots\leq i_{L_{1}}}}>\underset{L_{2}}{\underbrace{i_{L_{1}+1}\leq\cdots\leq i_{L_{1}+L_{2}}}}>\cdots>\underset{L_{k}}{\underbrace{i_{L_{1}+\cdots+L_{k-1}+1}\leq\cdots\leq i_{n}}.}
\]

The next theorem, due to Foulkes \cite{Foulkes1976} (see also \cite[Theorem 5]{Gessel1984}
and \cite[Corollary 7.23.8]{Stanley2001}), will play a pivotal role
in our approach to two-sided descent statistics.
\begin{thm}
\label{t-foulkes}Let $L$ and $M$ be compositions. Then $\left\langle r_{L},r_{M}\right\rangle $
is the number of permutations $\pi$ with descent composition $L$
such that $\pi^{-1}$ has descent composition $M$.
\end{thm}

Foulkes's theorem is a special case of a more general
theorem of Gessel on quasisymmetric generating functions, which we
briefly describe below. For a composition $L=(L_{1},L_{2},\dots,L_{k})$,
let $\Des(L)\coloneqq\{L_{1},L_{1}+L_{2},\dots,L_{1}+\cdots+L_{k-1}\}$,
and recall that the fundamental quasisymmetric function $F_{L}$ is
defined by 
\[
F_{L}\coloneqq\sum_{\substack{i_{1}\leq i_{2}\leq\cdots\leq i_{n}\\
i_{j}<i_{j+1}\text{ if }j\in\Des(L)
}
}x_{i_{1}}x_{i_{2}}\cdots x_{i_{n}}.
\]
Moreover, given a set $\Pi$ of permutations, its \textit{quasisymmetric
generating function} $Q(\Pi)$ is defined by 
\[
Q(\Pi)\coloneqq\sum_{\pi\in\Pi}F_{\Comp(\pi)}.
\]
The following is Corollary 4 of Gessel \cite{Gessel1984}.
\begin{thm}
\label{t-qsymgf}Suppose that $Q(\Pi)$ is a symmetric function. Then
the number of permutations in $\Pi$ with descent composition $L$
is equal to $\left\langle r_{L},Q(\Pi)\right\rangle $.
\end{thm}

Because $r_{M}$ is the quasisymmetric generating function for permutations
whose inverse has descent composition $M$, Foulkes's theorem follows
from Theorem \ref{t-qsymgf}. 

Gessel and Reutenauer \cite{Gessel1993}
showed that, for any partition $\lambda$, the quasisymmetric generating
function for permutations with cycle type $\lambda$ is a symmetric
function, and they used this fact along with Theorem \ref{t-qsymgf}
to study the joint distribution of $\maj$ and $\des$ over sets of
permutations with restricted cycle structure, including cyclic permutations,
involutions, and derangements. The present authors later studied the
distributions of $(\pk,\des)$, $(\lpk,\des)$, and $\udr$ over permutations
with restricted cycle structure \cite{Gessel2020}. Our current
work is a continuation of this line of research but for inverse descent
classes.

\subsection{Plethysm}

Let $A$ be a $\mathbb{Q}$-algebra of formal power series (possibly
containing $\Lambda$). Consider the operation $\Lambda\times A\rightarrow A$,
where the image of $(f,a)\in\Lambda\times A$ is denoted $f[a]$,
defined by the following two properties:
\begin{enumerate}
\item For any $i\geq1$, $p_{i}[a]$ is the result of replacing each variable
in $a$ with its $i$th power.
\item For any fixed $a\in A$, the map $f\mapsto f[a]$ is a $\mathbb{Q}$-algebra
homomorphism from $\Lambda$ to $A$.
\end{enumerate}
In other words, we have $p_{i}[f(x_{1},x_{2},\dots)]=f(x_{1}^{i},x_{2}^{i},\dots)$
for any $f\in\Lambda$. If $f$ contains variables other than the
$x_{i}$, then they are all raised to the $i$th power as well. For
example, if $q$ and $t$ are variables then $p_{i}[q^{2}tp_{m}]=q^{2i}t^{i}p_{im}$.
The map $(f,a)\mapsto f[a]$ is called \textit{plethysm}. As with
the scalar product, we extend plethysm in the obvious way to symmetric
functions of unbounded degree with coefficients involving other variables;
whenever we do so, we implicitly assume that any infinite sums involved
converge as formal power series so that the plethysms are defined.

We will need several technical lemmas involving plethysm in order
to evaluate the scalar products needed in our work. All of these lemmas
are from \cite{Gessel2020} by the present authors (or are easy consequences
of results from \cite{Gessel2020}). 

A \textit{monic term} is any monomial with coefficient 1.
\begin{lem}
\label{l-desmajhom}Let $\mathsf{m}\in A$ be a monic term not containing
any of the variables $x_{1},x_{2},\dots$. Then the maps $f\mapsto\left\langle f,H(\mathsf{m})\right\rangle $
and $f\mapsto\left\langle H(\mathsf{m}),f\right\rangle $ are $\mathbb{Q}$-algebra
homomorphisms on $A$.
\end{lem}

\begin{proof}
By a special case of \cite[Lemma 2.5]{Gessel2020}, we have
\[
f[\mathsf{m}]=\left\langle f,H(\mathsf{m})\right\rangle =\left\langle H(\mathsf{m}),f\right\rangle 
\]
and the result follows from the fact that $f\mapsto f[\mathsf{m}]$
is a $\mathbb{Q}$-algebra homomorphism. 
\end{proof}
\begin{lem}
\label{l-pkdeshom}Let $y\in A$ be a variable and $k\in\mathbb{Z}$.
Then the map $f\mapsto\left\langle f,E(y)^{k}H^{k}\right\rangle $
is a $\mathbb{Q}$-algebra homomorphism on $A$.
\end{lem}

\begin{proof}
Given $f\in A$ and a variable $\alpha\in A$, it is known \cite[Lemma 3.2]{Gessel2020}
that 
\[
f[k(1-\alpha)]=\left\langle f,E(-\alpha)^{k}H^{k}\right\rangle .
\]
Then the map $f\mapsto\left\langle f,E(y)^{k}H^{k}\right\rangle $
is obtained by composing the map $f\mapsto f[k(1-\alpha)]$ with evaluation
at $\alpha=-y$, both of which are $\mathbb{Q}$-algebra homomorphisms.
\end{proof}
\begin{lem}
\label{l-scalprodHE} Let $\alpha,\beta\in A$ be variables and let
$k$ be an integer. Then:
\begin{enumerate}
\item [\normalfont{(a)}] $H(\beta)[k(1-\alpha)]=\left\langle H(\beta),E(-\alpha)^{k}H^{k}\right\rangle =(1-\text{\ensuremath{\alpha}}\beta)^{k}/(1-\beta)^{k}$
\item [\normalfont{(b)}] $E(\beta)[k(1-\alpha)]=\left\langle E(\beta),E(-\alpha)^{k}H^{k}\right\rangle =(1+\beta)^{k}/(1+\alpha\beta)^{k}$
\end{enumerate}
\end{lem}

\begin{proof}
Part (a) is a special case of Lemma 2.4 (d) of \cite{Gessel2020}.
Part (b) follows immediately from part (a) and the well-known identity
$E(\beta)=(H(-\beta))^{-1}$.
\end{proof}

The two lemmas below are Lemmas 3.5 and 6.6 of \cite{Gessel2020}, respectively.
\begin{lem}
\label{l-sppleth}Let $f,g\in A$, and let $\mathsf{m}\in A$ be a monic term. Then 
$\left\langle f[X+\mathsf{m}], g\right\rangle$ = $\left\langle f, H[\mathsf{m}X]g\right\rangle$.
\end{lem}

\begin{lem}
\label{l-X+1}Let $\alpha\in A$ be a variable and $\mathsf{m}\in A$
a monic term. Then:
\begin{enumerate}
\item [\normalfont{(a)}] $H(\alpha)[X+\mathsf{m}]=H(\alpha)/(1-\alpha\mathsf{m})$
\item [\normalfont{(b)}] $E(\alpha)[X+\mathsf{m}]=(1+\alpha\mathsf{m})E(\alpha)$
\end{enumerate}
\end{lem}

\subsection{Symmetric function generating functions for descent statistics}

As mentioned in the introduction, we will need generating functions
for the ribbon skew Schur functions $r_{L}$ which keep track of the
descent statistics that we are studying. Such generating functions
were produced in previous work by the present authors and are stated
in the next lemma. Part (a) is essentially a commutative version of
\cite[Lemma 17]{Gessel2014}, whereas (b)\textendash (d) are given
in \cite[Lemma 2.8]{Gessel2020}.
\begin{lem}
\label{l-ribexp} We have \leqnomode
\begin{align*}
\tag{{a}}\qquad\sum_{n=0}^{\infty}t^{n}\prod_{k=0}^{n}H(q^{k}x) & =\sum_{n=0}^{\infty}\frac{\sum_{L\vDash n}q^{\maj(L)}t^{\des(L)}r_{L}}{(1-t)(1-qt)\cdots(1-q^{n}t)}x^{n},
\end{align*}
\begin{multline*}
\tag{{b}}\qquad\frac{1}{1-tE(yx)H(x)}=\frac{1}{1-t}+\\
\frac{1}{1+y}\sum_{n=1}^{\infty}\sum_{L\vDash n}\left(\frac{1+yt}{1-t}\right)^{n+1}\left(\frac{(1+y)^{2}t}{(y+t)(1+yt)}\right)^{\pk(L)+1}\left(\frac{y+t}{1+yt}\right)^{\des(L)+1}r_{L}x^{n},
\end{multline*}
\begin{multline*}
\tag{{c}}\qquad\frac{H(x)}{1-tE(yx)H(x)}=\frac{1}{1-t}+\\
\sum_{n=1}^{\infty}\sum_{L\vDash n}\frac{(1+yt)^{n}}{(1-t)^{n+1}}\left(\frac{(1+y)^{2}t}{(y+t)(1+yt)}\right)^{\lpk(L)}\left(\frac{y+t}{1+yt}\right)^{\des(L)}r_{L}x^{n},
\end{multline*}
and 
\begin{align*}
\tag{{d}}\qquad\frac{1+tH(x)}{1-t^{2}E(x)H(x)} & =\frac{1}{1-t}+\frac{1}{2(1-t)^{2}}\sum_{n=1}^{\infty}\sum_{L\vDash n}\frac{(1+t^{2})^{n}}{(1-t^{2})^{n-1}}\left(\frac{2t}{1+t^{2}}\right)^{\udr(L)}r_{L}x^{n}.
\end{align*}
\end{lem}

Some of these generating functions have nice power sum expansions
that are expressible in terms of Eulerian polynomials and type B Eulerian
polynomials. The $n$th \textit{type B Eulerian polynomial} $B_{n}(t)$
encodes the distribution of the type B descent number over the $n$th
hyperoctahedral group (see \cite[Section 2.3]{Zhuang2017} for definitions),
but can also be defined by the formula
\[
\frac{B_{n}(t)}{(1-t)^{n+1}}=\sum_{k=0}^{\infty}(2k+1)^{n}t^{k}
\]
analogous to \eqref{e-Eul}. 
Recall that $l(\lambda)$ is the number of parts of the partition
$\lambda$. We also define $o(\lambda)$ to be the number of odd parts
of $\lambda$, and use the notation $\sum_{\lambda\text{ odd}}$ for
a sum over partitions in which every part is odd.

The next lemma is \cite[Lemma 2.9]{Gessel2020}.
\begin{lem}
\label{l-psexp} We have \leqnomode
\[
\tag{{a}}\qquad\frac{1}{1-tE(yx)H(x)}=\sum_{\lambda}\frac{p_{\lambda}}{z_{\lambda}}\frac{A_{l(\lambda)}(t)}{(1-t)^{l(\lambda)+1}}x^{\left|\lambda\right|}\prod_{k=1}^{l(\lambda)}(1-(-y)^{\lambda_{k}})
\]
where $\lambda_{1},\lambda_{2},\dots,\lambda_{l(\lambda)}$ are the
parts of $\lambda$,
\[
\tag{{b}}\qquad\frac{H(x)}{1-tE(x)H(x)}=\sum_{\lambda}\frac{p_{\lambda}}{z_{\lambda}}\frac{B_{o(\lambda)}(t)}{(1-t)^{o(\lambda)+1}}x^{\left|\lambda\right|},
\]
and
\[
\tag{{c}}\qquad\frac{1+tH(x)}{1-tE(x)H(x)}=\sum_{\lambda\:\mathrm{odd}}\frac{p_{\lambda}}{z_{\lambda}}2^{l(\lambda)}\frac{A_{l(\lambda)}(t^{2})}{(1-t^{2})^{l(\lambda)+1}}x^{\left|\lambda\right|}+t\sum_{\lambda}\frac{p_{\lambda}}{z_{\lambda}}\frac{B_{o(\lambda)}(t^{2})}{(1-t^{2})^{o(\lambda)+1}}x^{\left|\lambda\right|}.
\]
\end{lem}

\subsection{Sums involving \texorpdfstring{$z_\lambda$}{z} and Stanley's formula for doubly alternating permutations}
\label{s-zsums}
The remainder of this section is not strictly about symmetric functions, but relates to the constants $z_\lambda$ which appear in symmetric function theory, and a connection to an enumeration formula of Stanley obtained via symmetric function techniques. 

To derive some of our formulas later on, we will need to evaluate several sums like
\begin{equation}
\label{e-zsum}
\sum_{\substack{\lambda\vdash n\\l(\lambda) =k}}\frac{1}{z_\lambda}.
\end{equation}
The key to evaluating these sums is the well-known fact that $n!/z_\lambda$ is the number of permutations in $\mathfrak{S}_n$ of cycle type $\lambda$; see, e.g., \cite[Proposition 1.3.2]{Stanley2011} and \cite[pp.~298--299]{Stanley2001}. Thus \eqref{e-zsum} is equal to $c(n,k)/n!$, where as before, $c(n,k)$ is the unsigned Stirling number of the first kind, which counts permutations in $\mathfrak{S}_n$ with $k$ cycles.

Now, let $d(n,k)$ be the number of permutations in $\mathfrak{S}_n$ with $k$ cycles, all of odd length; let $e(n,k)$ be the number of permutations in $\mathfrak{S}_n$ with $k$ odd cycles (and any number of even cycles); and let $f(n,k,m)$ be the number of permutations in $\mathfrak{S}_n$ with $k$ odd cycles and $m$ cycles in total. Note that $d(n,k)$, $e(n,k)$, and $f(n,k,m)$ are 0 if $n-k$ is odd. Then by the same reasoning as above, we have the additional sum evaluations in the next lemma. 
\begin{lem}
\label{l-cdef} We have \leqnomode
\begin{gather*}
\tag{a}\sum_{\substack{\lambda\vdash n\\l(\lambda) =k}}\frac{1}{z_\lambda}=\frac{c(n,k)}{n!},\\
\tag{b} 
\sum_{\substack{\lambda\vdash n\\
\mathrm{odd}\\
l(\lambda)=k
}} 
\frac{1}{z_{\lambda}}= \frac{d(n,k)}{n!},\\
\tag{c}
\sum_{\substack{\lambda\vdash n\\
o(\lambda)=k
}}
\frac{1}{z_{\lambda}}=\frac{e(n,k)}{n!},\\
\shortintertext{and}
\tag{d}
\sum_{\substack{\lambda\vdash n\\
o(\lambda)=k\\
l(\lambda)=m
}}
\frac{1}{z_{\lambda}}=\frac{f(n,k,m)}{n!}.
\end{gather*}
\end{lem}
The numbers $c(n,k)$, $d(n,k)$, $e(n,k)$, and $f(n,k,m)$ all have simple exponential generating functions. Let 
\begin{equation*}
L(u) \coloneqq \frac{1}{2} \log\frac{1+u}{1-u}=\sum_{m=1}^\infty \frac{u^{2m-1}}{2m-1}.
\end{equation*}

\begin{prop} 
\label{p-cdef}
\leqnomode
We have
\begin{gather*}
\tag{a}\sum_{n=0}^\infty \sum_{k=0}^n c(n,k) v^k \frac{u^n}{n!} = e^{-v\log(1-u)}=\frac{1}{(1-u)^v},\\
\tag{b}\sum_{n=0}^\infty \sum_{k=0}^n d(n,k) v^k \frac{u^n}{n!} 
   = e^{vL(u)}=\left(\frac{1+u}{1-u}\right)^{v/2},\\
\tag{c}\sum_{n=0}^\infty\sum_{k=0}^n e(n,k) v^k \frac{u^n}{n!}
   =\frac{e^{vL(u)}}{\sqrt{1-u^2}}=\frac{(1+u)^{(v-1)/2}}{(1-u)^{(v+1)/2}},\\
\shortintertext{and}
\tag{d}\sum_{n=0}^\infty\sum_{k,m=0}^n f(n,k,m) v^k w^m \frac{u^n}{n!}
   =\frac{e^{vwL(u)}}{(1-u^2)^{w/2}}
   =\frac{(1+u)^{(v-1)w/2}}{(1-u)^{(v+1)w/2}}.
\end{gather*}

\end{prop}
\begin{proof}
We prove only (c);  the proofs of the other formulas are similar (and (a) is well known). For (c) we want to count permutations in which odd cycles are weighted $v$ and even cycles are weighted 1. So by the exponential formula for permutations \cite[Corollary 5.1.9]{Stanley2001},
we have
\begin{align*}
\sum_{n=0}^\infty \sum_{k=0}^n e(n,k) \frac{u^n}{n!} v^k 
     &=\exp\biggl(v\sum_{m=1}^\infty \frac{u^{2m-1}}{2m-1} +\sum_{m=1}^\infty \frac{u^{2m}}{2m}\biggr)\\
     &=\exp\Bigl(\frac{v}{2}\bigl(-\log(1-u)+\log(1+u)\bigr)+\frac{1}{2}\bigl(-\log(1-u) -\log(1+u)\bigr)\Bigr)\\
       &=\frac{(1+u)^{(v-1)/2}}{(1-u)^{(v+1)/2}}.\qedhere
\end{align*}
\end{proof}

A permutation $\pi$ is called \textit{alternating} if
$\pi(1)>\pi(2)<\pi(3)>\pi(4)<\cdots$, and it is well known that
alternating permutations in $\mathfrak{S}_n$ are counted by the $n$th \textit{Euler number} $E_{n}$, whose exponential generating function is $\sum_{n=0}^\infty E_n x^n/n! = \sec x + \tan x$. The series $L(u)$ defined above makes an appearance in Stanley's work on \textit{doubly alternating permutations}: alternating permutations whose inverses are alternating.
Let $\tilde E_n$ denote the number of doubly alternating permutations in $\mathfrak{S}_n$.
Stanley showed that the ordinary generating function for doubly alternating permutations is
\begin{equation}
\label{e-altalt}
\sum_{n=0}^{\infty} \tilde E_n u^n =
\frac{1}{\sqrt{1-u^{2}}}\sum_{r=0}^{\infty}E_{2r}^{2}\frac{L(u)^{2r}}{(2r)!}+\sum_{r=0}^{\infty}E_{2r+1}^{2}\frac{L(u)^{2r+1}}{(2r+1)!}
\end{equation}
\cite[Theorem 3.1]{Stanley2007}. Stanley's proof of \eqref{e-altalt} uses Foulkes's theorem, but \eqref{e-altalt} can also be recovered from our results, as shall be demonstrated at the end of Sections \ref{ss-pkipk} and \ref{ss-lpkilpk}.

We note that Stanley's formula \eqref{e-altalt} can be expressed in terms of the numbers $d(n,k)$ and $e(n,k)$. By Proposition \ref{p-cdef}, we have
\begin{gather*}
\frac{1}{\sqrt{1-u^2}}\frac{L(u)^{2r}}{(2r)!} = \sum_{n=0}^\infty e(n, 2r) \frac{u^n}{n!}
\qquad\text{and}\qquad
\frac{L(u)^{2r+1}}{(2r+1)!} = \sum_{d=0}^\infty d(n, 2r+1) \frac{u^n}{n!}.
\end{gather*}
Thus, \eqref{e-altalt} is equivalent to the statement that
\begin{equation*}
\tilde E_n = 
\begin{cases}
\displaystyle\frac{1}{n!}\sum_r e(n,2r)E_{2r}^2 = \frac{1}{n!}\sum_{k=0}^n e(n,k) E_{k}^2,& \text{ if $n$ is even,}\\
\displaystyle\frac{1}{n!}\sum_r d(n, 2r+1)E_{2r+1}^2 =  \frac{1}{n!}\sum_{k=0}^n d(n,k) E_{k}^2,& \text{ if $n$ is odd.}
\end{cases}
\end{equation*}

\section{\label{s-pkdes}Two-sided peak and descent statistics}

We are now ready to proceed to the main body of our work. Our first
task will be to derive formulas for the two-sided distribution of
$(\pk,\des)$\textemdash i.e., the joint distribution of $\pk$, $\ipk$,
$\des$, and $\ides$ over $\mathfrak{S}_{n}$\textemdash and then
we shall specialize our results to the (mixed) two-sided statistics
$(\pk,\ipk)$, $(\pk,\ides)$, and $(\des,\ides)$.

\subsection{Peaks, descents, and their inverses}

Define the polynomials $P_{n}^{(\pk,\ipk,\des,\ides)}(y,z,s,t)$ by
\begin{align*}
P_{n}^{(\pk,\ipk,\des,\ides)}(y,z,s,t) & \coloneqq\sum_{\pi\in\mathfrak{S}_{n}}y^{\pk(\pi)+1}z^{\ipk(\pi)+1}s^{\des(\pi)+1}t^{\ides(\pi)+1},
\end{align*}
which encodes the desired two-sided distribution.
\begin{thm}
\label{t-pkdes}We have \leqnomode
\begin{multline*}
\tag{{a}}\frac{1}{(1-s)(1-t)}+\sum_{n=1}^{\infty}\frac{\left(\frac{(1+ys)(1+zt)}{(1-s)(1-t)}\right)^{n+1}P_{n}^{(\pk,\ipk,\des,\ides)}\left(\frac{(1+y)^{2}s}{(y+s)(1+ys)},\frac{(1+z)^{2}t}{(z+t)(1+zt)},\frac{y+s}{1+ys},\frac{z+t}{1+zt}\right)}{(1+y)(1+z)}x^{n}\\
=\sum_{i,j=0}^{\infty}\left(\frac{(1+yx)(1+zx)}{(1-yzx)(1-x)}\right)^{ij}s^{i}t^{j}
\end{multline*}
and, for all $n\geq1$, we have
\begin{multline*}
\tag{{b}}\frac{\left(\frac{(1+ys)(1+zt)}{(1-s)(1-t)}\right)^{n+1}P_{n}^{(\pk,\ipk,\des,\ides)}\left(\frac{(1+y)^{2}s}{(y+s)(1+ys)},\frac{(1+z)^{2}t}{(z+t)(1+zt)},\frac{y+s}{1+ys},\frac{z+t}{1+zt}\right)}{(1+y)(1+z)}\\
=\sum_{\lambda\vdash n}\frac{\prod_{i=1}^{l(\lambda)}(1-(-y)^{\lambda_{i}})(1-(-z)^{\lambda_{i}})}{z_{\lambda}}\frac{A_{l(\lambda)}(s)A_{l(\lambda)}(t)}{(1-s)^{l(\lambda)+1}(1-t)^{l(\lambda)+1}}.
\end{multline*}
\end{thm}

\begin{proof}
To prove this theorem, we shall compute $\left\langle (1-sE(yx)H(x))^{-1},(1-tE(z)H)^{-1}\right\rangle $
in three different ways. First, from Lemma \ref{l-ribexp} (b) we
have 
\begin{alignat}{1}
 & \left\langle \frac{1}{1-sE(yx)H(x)},\frac{1}{1-tE(z)H}\right\rangle \nonumber \\
 & \qquad\qquad=\frac{1}{(1-s)(1-t)}+\sum_{m,n=1}^{\infty}\frac{\left(\frac{1+ys}{1-s}\right)^{m+1}\left(\frac{1+zt}{1-t}\right)^{n+1}\sum_{L\vDash m,\,M\vDash n}N_{L,M}\left\langle r_{L},r_{M}\right\rangle }{(1+y)(1+z)}x^{m}\label{e-pkdessc0}
\end{alignat}
where 
\begin{multline*}
N_{L,M}\coloneqq\\
\quad\left(\frac{(1+y)^{2}s}{(y+s)(1+ys)}\right)^{\pk(L)+1}\left(\frac{(1+z)^{2}t}{(z+t)(1+zt)}\right)^{\pk(M)+1}\left(\frac{y+s}{1+ys}\right)^{\des(L)+1}\left(\frac{z+t}{1+zt}\right)^{\des(M)+1}.
\end{multline*}
Upon applying Foulkes's theorem (Theorem \ref{t-foulkes}), Equation
(\ref{e-pkdessc0}) simplifies to {\allowdisplaybreaks
\begin{align}
 & \left\langle \frac{1}{1-sE(yx)H(x)},\frac{1}{1-tE(z)H}\right\rangle \nonumber \\
 & \;=\frac{1}{(1-s)(1-t)}+\sum_{n=1}^{\infty}\frac{\left(\frac{(1+ys)(1+zt)}{(1-s)(1-t)}\right)^{n+1}\sum_{\pi\in\mathfrak{S}_{n}}N_{\Comp(\pi),\Comp(\pi^{-1})}}{(1+y)(1+z)}x^{n}\nonumber \\
 & \;=\frac{1}{(1-s)(1-t)}+\sum_{n=1}^{\infty}\frac{\left(\frac{(1+ys)(1+zt)}{(1-s)(1-t)}\right)^{n+1}P_{n}^{(\pk,\ipk,\des,\ides)}\left(\frac{(1+y)^{2}s}{(y+s)(1+ys)},\frac{(1+z)^{2}t}{(z+t)(1+zt)},\frac{y+s}{1+ys},\frac{z+t}{1+zt}\right)}{(1+y)(1+z)}x^{n}.\label{e-pkdessc1}
\end{align}
}Second, we have 
\begin{alignat}{1}
\left\langle \frac{1}{1-sE(yx)H(x)},\frac{1}{1-tE(z)H}\right\rangle  & =\left\langle \sum_{i=0}^{\infty}E(yx)^{i}H(x)^{i}s^{i},\sum_{j=0}^{\infty}E(z)^{j}H^{j}t^{j}\right\rangle \nonumber \\
 & =\sum_{i,j=0}^{\infty}\left\langle E(yx)^{i}H(x)^{i},E(z)^{j}H^{j}\right\rangle s^{i}t^{j}\nonumber \\
 & =\sum_{i,j=0}^{\infty}\left\langle E(yx),E(z)^{j}H^{j}\right\rangle ^{i}\left\langle H(x),E(z)^{j}H^{j}\right\rangle ^{i}s^{i}t^{j}\nonumber \\
 & =\sum_{i,j=0}^{\infty}\frac{(1+yx)^{ij}}{(1-yzx)^{ij}}\frac{(1+zx)^{ij}}{(1-x)^{ij}}s^{i}t^{j},\label{e-pkdessc2}
\end{alignat}
where the last two steps are obtained using Lemmas \ref{l-pkdeshom}
and \ref{l-scalprodHE}, respectively. Finally, from Lemma \ref{l-psexp}
(a) we have {\allowdisplaybreaks
\begin{align}
 & \left\langle \frac{1}{1-sE(yx)H(x)},\frac{1}{1-tE(z)H}\right\rangle \nonumber \\
 & \qquad\qquad=\sum_{\lambda,\mu}\frac{\prod_{i=1}^{l(\lambda)}(1-(-y)^{\lambda_{i}})\prod_{j=1}^{l(\mu)}(1-(-z)^{\mu_{i}})}{z_{\lambda}z_{\mu}}\frac{A_{l(\lambda)}(s)A_{l(\mu)}(t)}{(1-s)^{l(\lambda)+1}(1-t)^{l(\mu)+1}}\left\langle p_{\lambda},p_{\mu}\right\rangle x^{\left|\lambda\right|}\nonumber \\
 & \qquad\qquad=\sum_{\lambda}\frac{\prod_{i=1}^{l(\lambda)}(1-(-y)^{\lambda_{i}})(1-(-z)^{\lambda_{i}})}{z_{\lambda}}\frac{A_{l(\lambda)}(s)A_{l(\lambda)}(t)}{(1-s)^{l(\lambda)+1}(1-t)^{l(\lambda)+1}}x^{\left|\lambda\right|}.\label{e-pkdessc3}
\end{align}
}Combining (\ref{e-pkdessc1}) with (\ref{e-pkdessc2}) yields part
(a), whereas combining (\ref{e-pkdessc1}) with (\ref{e-pkdessc3})
and then extracting coefficients of $x^{n}$ yields part (b).
\end{proof}
The formulas in Theorem \ref{t-pkdes} and others appearing later
in this paper can be ``inverted'' upon making appropriate substitutions.
For example, Theorem \ref{t-pkdes} (a) can be rewritten as
\begin{multline*}
\frac{1}{(1-\alpha)(1-\beta)}+\frac{1}{(1+u)(1+v)}\sum_{n=1}^{\infty}\left(\frac{(1+u\alpha)(1+v\beta)}{(1-\alpha)(1-\beta)}\right)^{n+1}P_{n}^{(\pk,\ipk,\des,\ides)}(y,z,s,t)x^{n}\\
=\sum_{i,j=0}^{\infty}\left(\frac{(1+ux)(1+vx)}{(1-uvx)(1-x)}\right)^{ij}\alpha^{i}\beta^{j}
\end{multline*}
where{\allowdisplaybreaks
\begin{align*}
u&=\frac{1+s^{2}-2ys-(1-s)\sqrt{(1+s)^{2}-4ys}}{2(1-y)s},\\ 
\alpha&=\frac{(1+s)^{2}-2ys-(1+s)\sqrt{(1+s)^{2}-4ys}}{2ys},\\
v&=\frac{1+t^{2}-2zt-(1-t)\sqrt{(1+t)^{2}-4zt}}{2(1-z)t}, \\
\shortintertext{and} 
\beta&=\frac{(1+t)^{2}-2zt-(1+t)\sqrt{(1+t)^{2}-4zt}}{2zt}.
\end{align*}}

\subsection{Peaks and inverse peaks} \label{ss-pkipk}

We will now consider specializations of the polynomials $P_{n}^{(\pk,\ipk,\des,\ides)}(y,z,s,t)$
that give the distributions of the (mixed) two-sided statistics $(\pk,\ipk)$,
$(\pk,\ides)$, and $(\des,\ides)$. Let us begin with the two-sided
distribution of $\pk$, which is encoded by
\begin{align*}
P_{n}^{(\pk,\ipk)}(s,t) & \coloneqq P_{n}^{(\pk,\ipk,\des,\ides)}(s,t,1,1)=\sum_{\pi\in\mathfrak{S}_{n}}s^{\pk(\pi)+1}t^{\ipk(\pi)+1}.
\end{align*}

\begin{thm}
\label{t-pk}We have \leqnomode
\begin{multline*}
\tag{{a}}\frac{1}{(1-s)(1-t)}+\frac{1}{4}\sum_{n=1}^{\infty}\left(\frac{(1+s)(1+t)}{(1-s)(1-t)}\right)^{n+1}P_{n}^{(\pk,\ipk)}\left(\frac{4s}{(1+s)^{2}},\frac{4t}{(1+t)^{2}}\right)x^{n}\\
=\sum_{i,j=0}^{\infty}\left(\frac{1+x}{1-x}\right)^{2ij}s^{i}t^{j}
\end{multline*}
and, for all $n\geq1$, we have
\begin{multline*}
\tag{{b}}\qquad\left(\frac{(1+s)(1+t)}{(1-s)(1-t)}\right)^{n+1}P_{n}^{(\pk,\ipk)}\left(\frac{4s}{(1+s)^{2}},\frac{4t}{(1+t)^{2}}\right)\\
=4\sum_{i,j=0}^{\infty}\sum_{k=0}^{n}{2ij \choose k}{2ij+n-k-1 \choose n-k}s^{i}t^{j}\qquad
\end{multline*}
and
\begin{multline*}
\tag{{c}}\qquad\left(\frac{(1+s)(1+t)}{(1-s)(1-t)}\right)^{n+1}P_{n}^{(\pk,\ipk)}\left(\frac{4s}{(1+s)^{2}},\frac{4t}{(1+t)^{2}}\right)\\
=\frac{1}{n!}\sum_{k=0}^n
4^{k+1}d(n,k)\frac{A_{k}(s)A_{k}(t)}{(1-s)^{k+1}(1-t)^{k+1}}\qquad
\end{multline*}
with $d(n,k)$ as defined in Section \ref{s-zsums}.
\end{thm}
\begin{proof}
Parts (a) and (c) are obtained from evaluating Theorem \ref{t-pkdes}
(a) and (b), respectively, at $y=z=1$, with the sum on $\lambda$ in (c) evaluated by Lemma \ref{l-cdef} (b). From the identities
\[
(1+x)^{k}=\sum_{n=0}^{k}{k \choose n}x^{n}\quad\text{and}\quad(1-x)^{-k}=\sum_{n=0}^{\infty}{k+n-1 \choose n}x^{n}
\]
we obtain 
\begin{equation}
\left(\frac{1+x}{1-x}\right)^{2ij}=\sum_{n=0}^{\infty}\sum_{k=0}^{n}{2ij \choose k}{2ij+n-k-1 \choose n-k}x^{n}.\label{e-2ij}
\end{equation}
Substituting (\ref{e-2ij}) into (a) and extracting coefficients of
$x^{n}$ yields (b).
\end{proof}
\renewcommand{\arraystretch}{1.3}

\begin{table}[H]
\noindent \begin{raggedright}
\centering
{\small{}}%
\begin{tabular}{|>{\centering}m{0.1in}|>{\raggedright}p{6in}|}
\hline 
\multicolumn{1}{|c|}{{\small{}$n$}} & {\small{}$P_{n}^{(\pk,\ipk)}(s,t)$}\tabularnewline
\hline 
{\small{}$1$} & {\small{}$st$}\tabularnewline
\hline 
{\small{}$2$} & {\small{}$2st$}\tabularnewline
\hline 
{\small{}$3$} & {\small{}$(3s+s^{2})t+(s+s^{2})t^{2}$}\tabularnewline
\hline 
{\small{}$4$} & {\small{}$(4s+4s^{2})t+(4s+12s^{2})t^{2}$}\tabularnewline
\hline 
{\small{}$5$} & {\small{}$(5s+10s^{2}+s^{3})t+(10s+66s^{2}+12s^{3})t^{2}+(s+12s^{2}+3s^{3})t^{3}$}\tabularnewline
\hline 
{\small{}$6$} & {\small{}$(6s+20s^{2}+6s^{3})t+(20s+248s^{2}+148s^{3})t^{2}+(6s+148s^{2}+118s^{3})t^{3}$}\tabularnewline
\hline 
{\small{}$7$} & {\small{}$(7s+35s^{2}+21s^{3}+s^{4})t+(35s+739s^{2}+969s^{3}+81s^{4})t^{2}+(21s+969s^{2}+1719s^{3}+141s^{4})t^{3}+(s+81s^{2}+171s^{3}+19s^{4})t^{4}$}\tabularnewline
\hline 
\end{tabular}{\small\par}
\par\end{raggedright}
\caption{\label{tb-pk}Joint distribution of $\protect\pk$ and $\protect\ipk$
over $\mathfrak{S}_{n}$}
\end{table}

The first several polynomials $P_{n}^{(\pk,\ipk)}(s,t)$ are displayed
in Table \ref{tb-pk}. The coefficients of $s^{k}t$\textemdash equivalently,
the coefficients of $st^{k}$ by symmetry\textemdash are characterized
by the following proposition.
\begin{prop}
\label{p-pkipk0}For any $n\geq1$ and $k\geq0$, the number of permutations
$\pi\in\mathfrak{S}_{n}$ with $\pk(\pi)=k$ and $\ipk(\pi)=0$ is
equal to ${n \choose 2k+1}$.
\end{prop}

Proposition \ref{p-pkipk0} was first stated as a corollary of a more
general result of Troyka and Zhuang \cite{Troyka2022}, namely, that
for any composition $L$ of $n\geq1$, there is exactly one permutation
in $\pi\in\mathfrak{S}_{n}$ with descent composition $L$ such that $\ipk(\pi)=0$.
However, it is also possible to prove Proposition \ref{p-pkipk0} directly
from Theorem \ref{t-pk}. 

Next, we obtain a surprisingly simple formula expressing the two-sided
peak polynomials $P_{n}^{(\pk,\ipk)}(s,t)$ in terms of products of
the ordinary peak polynomials
\[
P_{n}^{\pk}(t)\coloneqq\sum_{\pi\in\mathfrak{S}_{n}}t^{\pk(\pi)+1}
\]
giving the distribution of the peak number over $\mathfrak{S}_{n}$.

\begin{thm}
\label{t-pkprod}For all $n\geq1$, we have 
\begin{align*}
P_{n}^{(\pk,\ipk)}(s,t) &= \frac{1}{n!}\sum_{k=0}^n d(n,k) \left((1-s)(1-t)\right)^{\frac{n-k}{2}}P_{k}^{\pk}(s)P_{k}^{\pk}(t). 
\end{align*}
\end{thm}

Since $d(n,k)=0$ when $n-k$ is odd, the above formula does not involve any square roots.

\begin{proof}
It is known \cite{Stembridge1997} that 
\begin{equation}
A_{n}(t)=\left(\frac{1+t}{2}\right)^{n+1}P_{n}^{\pk}\left(\frac{4t}{(1+t)^{2}}\right)\label{e-Apk}
\end{equation}
for  $n\geq1$. Combining (\ref{e-Apk}) with Theorem \ref{t-pk} (c) and then replacing the variables $s$ and $t$ with $u$ and $v$, respectively, yields 
\begin{multline*}
\quad P_{n}^{(\pk,\ipk)}\left(\frac{4u}{(1+u)^{2}},\frac{4v}{(1+v)^{2}}\right)\\
=\frac{1}{n!}\sum_{k=0}^n d(n,k)\left(\frac{(1-u)(1-v)}{(1+u)(1+v)}\right)^{n-k}P_{k}^{\pk}\left(\frac{4u}{(1+u)^{2}}\right)P_{k}^{\pk}\left(\frac{4v}{(1+v)^{2}}\right).
\end{multline*}
Setting $s=4u/(1+u)^{2}$ and $t=4v/(1+v)^{2}$, we obtain
\begin{equation}
P_{n}^{(\pk,\ipk)}(s,t)
 =\frac{1}{n!}\sum_{k=0}^n d(n,k) \left(\frac{(1-u)(1-v)}{(1+u)(1+v)}\right)^{n-k}P_{k}^{\pk}(s)P_{k}^{\pk}(t).\label{e-pkprod}
\end{equation}
It can be verified that $u=2s^{-1}(1-\sqrt{1-s})-1$ and $v=2t^{-1}(1-\sqrt{1-t})-1$, which lead to
\begin{equation}
\frac{(1-u)(1-v)}{(1+u)(1+v)}=\sqrt{(1-s)(1-t)}.\label{e-pksimp}
\end{equation}
Substituting (\ref{e-pksimp}) into (\ref{e-pkprod}) completes the proof.
\end{proof}

The substitution used in the proof of Theorem \ref{t-pkprod} allows
us to invert the formulas in Theorem \ref{t-pk} and others involving
the same quadratic transformation. For example, Theorem~\ref{t-pk}
(b) is equivalent to
\begin{alignat*}{1}
\frac{P_{n}^{(\pk,\ipk)}(s,t)}{(1-s)^{\frac{n+1}{2}}(1-t)^{\frac{n+1}{2}}} & =4\sum_{i,j=0}^{\infty}\sum_{k=0}^{n}{2ij \choose k}{2ij+n-k-1 \choose n-k}u^{i}v^{j}
\end{alignat*}
where, as before, $u=2s^{-1}(1-\sqrt{1-s})-1$ and $v=2t^{-1}(1-\sqrt{1-t})-1$.
In fact, we can express $u$ and $v$ in terms of the Catalan generating
function
\[
C(x) \coloneqq \frac{1-\sqrt{1-4x}}{2x}=\sum_{n=0}^{\infty}\frac{1}{n+1}{2n \choose n}x^{n}
\]
as $u=(s/4)C(s/4)^{2}$ and $v=(t/4)C(t/4)^{2}$.

If we multiply the formula of Theorem \ref{t-pkprod} by $u^n$ and sum over $n$ using Proposition \ref{p-cdef} (b), we obtain the following generating function for the two-sided peak polynomials $P_{n}^{(\pk,\ipk)}(s,t)$.

\begin{thm}
\label{t-pkipkgf}
We have
\begin{equation*}
\sum_{n=0}^\infty P_{n}^{(\pk,\ipk)}(s,t) u^n = \sum_{k=0}^\infty \frac{1}{k!} \left(\frac{L(\sqrt{(1-s)(1-t)}u)}{\sqrt{(1-s)(1-t)}}\right)^k P_{k}^{\pk}(s)P_{k}^{\pk}(t),
\end{equation*}
with $L(u)$ as defined in Section \ref{s-zsums}.
\end{thm}

We can derive the odd part of Stanley's generating function \eqref{e-altalt} for doubly alternating permutations from Theorem \ref{t-pkipkgf}. While an alternating permutation $\pi$ satisfies $\pi(1)>\pi(2)<\pi(3)>\pi(4)<\cdots$, we say that $\pi$ is \textit{reverse alternating} if $\pi(1)<\pi(2)>\pi(3)<\pi(4)>\cdots$. It is evident by symmetry that reverse alternating permutations are also counted by the Euler numbers $E_n$. As shown by Stanley \cite{Stanley2007}, the number of doubly alternating permutations in $\mathfrak{S}_n$ is equal to the number $\tilde{E}_n$ of reverse alternating permutations in $\mathfrak{S}_n$ whose inverses are reverse alternating.

It is readily verified that a permutation $\pi$ in $\mathfrak{S}_n$ has at most $(n-1)/2$ peaks, and has exactly $(n-1)/2$ peaks if and only if $n$ is odd and $\pi$ is reverse alternating. 
Thus, we have
\[
\lim_{s\rightarrow0}P_{k}^{\pk}(s^{-2})s^{k+1}=\lim_{s\rightarrow0}\sum_{\pi\in\mathfrak{S}_{k}}s^{k-2\pk(\pi)-1}=\begin{cases}
0, & \text{if $k$ is even,}\\
E_{k}, & \text{if $k$ is odd.}
\end{cases}
\]
Similarly, we have
\[
\lim_{s\rightarrow0}P_{n}^{(\pk,\ipk)}(s^{-2},t^{-2})s^{n+1}t^{n+1}=\begin{cases}
0, & \text{if $n$ is even,}\\
\tilde{E}_{n}, & \text{if $n$ is odd.}
\end{cases}
\]
So, to obtain Stanley's formula from Theorem \ref{t-pkipkgf}, we first replace $s$ with $s^{-2}$, $t$ with $t^{-2}$, and $u$ with $stu$; then we multiply by $st$ and take the limit as $s,t\to 0$. The substitution takes 
\[\left(\frac{L(\sqrt{(1-s)(1-t)}u)}{\sqrt{(1-s)(1-t)}}\right)^k P_{k}^{\pk}(s)P_{k}^{\pk}(t)
\]
to 
\begin{multline*}
\qquad\left(\frac{L(\sqrt{(s^{2}-1)(t^{2}-1)}u)}{\sqrt{(1-s^{-2})(1-t^{-2})}}\right)^{k}P_{k}^{\pk}(s^{-2})P_{k}^{\pk}(t^{-2})\\
=\left(\frac{L(\sqrt{(s^{2}-1)(t^{2}-1)}u)}{\sqrt{(s^{2}-1)(t^{2}-1)}}\right)^{k}s^{k}P_{k}^{\pk}(s^{-2})t^{k}P_{k}^{\pk}(t^{-2});\qquad
\end{multline*}
multiplying by $st$ and taking $s,t\to 0$ gives $L(u)^k E_k^2$ for $k$ odd and $0$ for $k$ even, as desired.

We will later get the even part of Stanley's generating function \eqref{e-altalt} in a similar way from Theorem \ref{t-lpkilpkprod}, which is the analogue of Theorem \ref{t-pkipkgf} for left peaks.

Before proceeding, we note that every formula like Theorem \ref{t-pkprod} has an analogous generating function like Theorem \ref{t-pkipkgf}. We will only write out these generating functions for 
$(\pk, \ipk)$ and $(\lpk,\ilpk)$---which are used in our rederivation of Stanley's formula---as well as for $(\des,\ides)$.

\subsection{Peaks and inverse descents}

Next, define the polynomials $P_{n}^{(\pk,\ides)}(s,t)$ by 
\[
P_{n}^{(\pk,\ides)}(s,t)\coloneqq P_{n}^{(\pk,\ipk,\des,\ides)}(s,1,1,t)=\sum_{\pi\in\mathfrak{S}_{n}}s^{\pk(\pi)+1}t^{\ides(\pi)+1}.
\]
We omit the proof of the next theorem as it is similar to that of Theorem
\ref{t-pk}, with the main difference being that we specialize at
$y=1$ and $z=0$ (as opposed to $y=z=1$).
\begin{thm}
\label{t-pkides}We have \leqnomode 
\begin{multline*}
\tag{{a}}\qquad\frac{1}{(1-s)(1-t)}+\frac{1}{2}\sum_{n=1}^{\infty}\left(\frac{1+s}{(1-s)(1-t)}\right)^{n+1}P_{n}^{(\pk,\ides)}\left(\frac{4s}{(1+s)^{2}},t\right)\\
=\sum_{i,j=0}^{\infty}\left(\frac{1+x}{1-x}\right)^{ij}s^{i}t^{j}\qquad
\end{multline*}
and, for all $n\geq1$, we have
\begin{align*}
\tag{{b}}\qquad\left(\frac{1+s}{(1-s)(1-t)}\right)^{n+1}P_{n}^{(\pk,\ides)}\left(\frac{4s}{(1+s)^{2}},t\right) & =2\sum_{i,j=0}^{\infty}\sum_{k=0}^{n}{ij \choose k}{ij+n-k-1 \choose n-k}s^{i}t^{j}
\end{align*}
and
\begin{align*}
\tag{{c}} \quad\, \left(\frac{1+t}{(1-s)(1-t)}\right)^{\! n+1}P_{n}^{(\pk,\ides)}\left(\frac{4s}{(1+s)^{2}},t\right) 
 &=\frac{1}{n!}\sum_{k=0}^n 2^{k+1} d(n,k)
 \frac{A_{k}(s)A_{k}(t)}{(1-s)^{k+1}(1-t)^{k+1}}.
\end{align*}
\end{thm}

We display the first several polynomials $P_{n}^{(\pk,\ides)}(s,t)$
in Table \ref{tb-desipk}, collecting terms with the same power of
$s$ in order to display the symmetry in their coefficients. This
symmetry follows from the fact that the \textit{reverse} $\pi^{r}\coloneqq\pi(n)\cdots\pi(2)\pi(1)$
of a permutation $\pi\in\mathfrak{S}_{n}$ satisfies $\ides(\pi^{r})=n-1-\ides(\pi)$
\cite[Proposition 2.6 (c)]{Zhuang2022} but has the same number of
peaks as $\pi$. 

\renewcommand{\arraystretch}{1.3}
\begin{table}[H]
\centering
\noindent \begin{raggedright}
{\small{}}%
\begin{tabular}{|>{\centering}m{0.1in}|>{\raggedright}p{6in}|}
\hline 
\multicolumn{1}{|c|}{{\small{}$n$}} & {\small{}$P_{n}^{(\pk,\ides)}(s,t)$}\tabularnewline
\hline 
{\small{}$1$} & {\small{}$ts$}\tabularnewline
\hline 
{\small{}$2$} & {\small{}$(t+t^{2})s$}\tabularnewline
\hline 
{\small{}$3$} & {\small{}$(t+2t^{2}+t^{3})s+2t^{2}s^{2}$}\tabularnewline
\hline 
{\small{}$4$} & {\small{}$(t+3t^{2}+3t^{3}+t^{4})s+(8t^{2}+8t^{3})s^{2}$}\tabularnewline
\hline 
{\small{}$5$} & {\small{}$(t+4t^{2}+6t^{3}+4t^{4}+t^{5})s+(20t^{2}+48t^{3}+20t^{4})s^{2}+(2t^{2}+12t^{3}+12t^{4})s^{3}$}\tabularnewline
\hline 
{\small{}$6$} & {\small{}$(t+5t^{2}+10t^{3}+10t^{4}+5t^{5}+t^{6})s+(40t^{2}+168t^{3}+168t^{4}+40t^{5})s^{2}+(12t^{2}+124t^{3}+124t^{4}+12t^{5})s^{3}$}\tabularnewline
\hline 
{\small{}$7$} & {\small{}$(t+6t^{2}+15t^{3}+20t^{4}+15t^{5}+6t^{6}+t^{7})s+(70t^{2}+448t^{3}+788t^{4}+448t^{5}+70t^{6})s^{2}+(42t^{2}+672t^{3}+1452t^{4}+672t^{5}+42t^{6})s^{3}+(2t^{2}+56t^{3}+156t^{4}+56t^{5}+2t^{6})s^{4}$}\tabularnewline
\hline 
\end{tabular}{\small\par}
\par\end{raggedright}
\caption{\label{tb-desipk}Joint distribution of $\protect\pk$ and $\protect\ides$
over $\mathfrak{S}_{n}$}
\end{table}

Furthermore, observe that the coefficients of $t^{k}s$
in $P_{n}^{(\pk,\ides)}(s,t)$ are binomial coefficients; this is
a consequence of the following proposition, which is Corollary 9 of
\cite{Troyka2022}.
\begin{prop}
\label{p-pkides0}For any $n\geq1$ and $k\geq0$, the number of permutations
$\pi\in\mathfrak{S}_{n}$ with $\des(\pi)=k$ and $\ipk(\pi)=0$ is
equal to ${n-1 \choose k}$.
\end{prop}

In addition to being symmetric, the coefficients of $s^{k}$ in $P_{n}^{(\pk,\ides)}(s,t)$
seem to be unimodal polynomials in $t$, which we know holds for $k=0$
in light of Proposition \ref{p-pkides0}. In the last section of this
paper, we will state the conjecture that the polynomials $[s^{k}]\,P_{n}^{(\pk,\ides)}(s,t)$
are in fact $\gamma$-positive, a property which implies unimodality
and symmetry.

The next formula expresses $P_{n}^{(\pk,\ides)}(s,t)$
in terms of the products $P_{k}^{\pk}(s)A_{k}(t)$. The proof is very
similar to that of Theorem \ref{t-pkprod} and so it is omitted.
\begin{thm}
\label{t-pkidesprod} For all $n\geq1$, we have
\begin{align*}
P_{n}^{(\pk,\ides)}(s,t) & =\frac{1}{n!}\sum_{k=0}^n d(n,k)\left((1-s)^{1/2}(1-t)\right)^{n-k}P_{k}^{\pk}(s)A_{k}(t).
\end{align*}
\end{thm}

\subsection{Descents and inverse descents}

To conclude this section, we state the analogous results for the
two-sided Eulerian polynomials
\begin{align*}
A_{n}(s,t) & =P_{n}^{(\pk,\ipk,\des,\ides)}(1,1,s,t)=\sum_{\pi\in\mathfrak{S}_{n}}s^{\des(\pi)+1}t^{\ides(\pi)+1}.
\end{align*}
Recall that parts (a) and (b) were originally due to Carlitz, Roselle, and Scoville \cite{Carlitz1966},
whereas (c) and (d) are new.

\begin{thm}
\label{t-des}We have \leqnomode 
\begin{alignat*}{1}
\tag{{a}}\qquad\sum_{n=0}^{\infty}\frac{A_{n}(s,t)}{(1-s)^{n+1}(1-t)^{n+1}}x^{n} & =\sum_{i,j=0}^{\infty}\frac{s^{i}t^{j}}{(1-x)^{ij}}
\end{alignat*}
and, for all $n\geq1$, we have
\begin{align*}
\tag{{b}}\qquad\frac{A_{n}(s,t)}{(1-s)^{n+1}(1-t)^{n+1}} & =\sum_{i,j=0}^{\infty}{ij+n-1 \choose n}s^{i}t^{j}
\end{align*}
and
\[
\tag{{c}}\qquad A_{n}(s,t)
=\frac{1}{n!}\sum_{k=0}^n c(n,k) \left((1-s)(1-t)\right)^{n-k}A_{k}(s)A_{k}(t).
\]
Moreover, we have
\[
\tag{{d}}\qquad
\sum_{n=0}^\infty A_n(s,t) u^n 
  = \sum_{k=0}^\infty \frac{1}{k!}\left(\log\frac{1}{1-(1-s)(1-t)u}\right)^k \frac{A_k(s)A_k(t)}{(1-s)^k(1-t)^k}.
\]
\end{thm}

Parts (a)--(c) are proven similarly to Theorems \ref{t-pk}, \ref{t-pkprod}, and \ref{t-pkides}
except that we evaluate Theorem~\ref{t-pkdes} at $y=z=0$. Part (c) can also be derived directly from (b):
\begin{align*}
\frac{A_{n}(s,t)}{(1-s)^{n+1}(1-t)^{n+1}}
   &=\sum_{i,j=0}^{\infty}{ij+n-1 \choose n}s^{i}t^{j}\\
   &=\sum_{i,j=0}^{\infty}\frac{1}{n!}\sum_{k=0}^n c(n,k) (ij)^k s^{i}t^{j}\\
   &=\frac{1}{n!}\sum_{k=0}^n c(n,k)\sum_{i=0}^\infty i^k s^i \sum_{j=0}^\infty j^k t^j\\
   &=\frac{1}{n!}\sum_{k=0}^n c(n,k)\frac{A_k(s)}{(1-t)^{k+1}}\frac{A_k(t)}{(1-t)^{k+1}}.
\end{align*}
Part (d) is obtained from (c) using Proposition \ref{p-cdef} (a).

\section{\label{s-lpk}Two-sided left peak statistics}

\subsection{Left peaks, descents, and their inverses}

In this section, we will examine (mixed) two-sided distributions involving
the left peak number $\lpk$. Let us first derive a generating function
formula for the polynomials
\begin{align*}
P_{n}^{(\lpk,\ilpk,\des,\ides)}(y,z,s,t) & \coloneqq\sum_{\pi\in\mathfrak{S}_{n}}y^{\lpk(\pi)}z^{\ilpk(\pi)}s^{\des(\pi)}t^{\ides(\pi)}
\end{align*}
giving the joint distribution of $\lpk$, $\ilpk$, $\des$, and $\ides$
over $\mathfrak{S}_{n}$.
\begin{thm}
\label{t-lpkdes}We have
\begin{multline*}
\frac{1}{(1-s)(1-t)}+\sum_{n=1}^{\infty}\frac{(1+ys)^{n}(1+zt)^{n}P_{n}^{(\lpk,\ilpk,\des,\ides)}\left(\frac{(1+y)^{2}s}{(y+s)(1+ys)},\frac{(1+z)^{2}t}{(z+t)(1+zt)},\frac{y+s}{1+ys},\frac{z+t}{1+zt}\right)}{(1-s)^{n+1}(1-t)^{n+1}}x^{n}\\
=\sum_{i,j=0}^{\infty}\frac{(1+yx)^{i(j+1)}(1+zx)^{(i+1)j}}{(1-yzx)^{ij}(1-x)^{(i+1)(j+1)}}s^{i}t^{j}.
\end{multline*}
\end{thm}

\begin{proof}
From Lemma \ref{l-ribexp} (c) we have
\begin{alignat}{1}
 & \left\langle \frac{H(x)}{1-sE(yx)H(x)},\frac{H}{1-tE(z)H}\right\rangle \nonumber \\
 & \qquad\qquad\qquad=\frac{1}{(1-s)(1-t)}+\sum_{m,n=1}^{\infty}\sum_{\substack{L\vDash m\\
M\vDash n
}
}\frac{(1+ys)^{m}(1+zt)^{n}}{(1-s)^{m+1}(1-t)^{n+1}}\acute{N}_{L,M}\left\langle r_{L},r_{M}\right\rangle x^{m}\label{e-lpkdessc0}
\end{alignat}
where 
\[
\acute{N}_{L,M}\coloneqq\left(\frac{(1+y)^{2}s}{(y+s)(1+ys)}\right)^{\lpk(L)}\left(\frac{(1+z)^{2}t}{(z+t)(1+zt)}\right)^{\ilpk(M)}\left(\frac{y+s}{1+ys}\right)^{\des(L)}\left(\frac{z+t}{1+zt}\right)^{\des(M)}.
\]
By Foulkes's theorem (Theorem \ref{t-foulkes}), Equation (\ref{e-lpkdessc0})
simplifies to
\begin{alignat}{1}
 & \left\langle \frac{1}{1-sE(yx)H(x)},\frac{1}{1-tE(z)H}\right\rangle \nonumber \\
 & \quad=\frac{1}{(1-s)(1-t)}+\sum_{n=1}^{\infty}\frac{(1+ys)^{n}(1+zt)^{n}\sum_{\pi\in\mathfrak{S}_{n}}\acute{N}_{\Comp(\pi),\Comp(\pi^{-1})}}{(1-s)^{n+1}(1-t)^{n+1}}x^{n}\nonumber \\
 & \quad=\frac{1}{(1-s)(1-t)}\nonumber \\
 & \qquad\quad+\sum_{n=1}^{\infty}\frac{(1+ys)^{n}(1+zt)^{n}P_{n}^{(\lpk,\ilpk,\des,\ides)}\left(\frac{(1+y)^{2}s}{(y+s)(1+ys)},\frac{(1+z)^{2}t}{(z+t)(1+zt)},\frac{y+s}{1+ys},\frac{z+t}{1+zt}\right)}{(1-s)^{n+1}(1-t)^{n+1}}x^{n}.\label{e-lpkdessc1}
\end{alignat}
We now compute the same scalar product in a different way. We have{\allowdisplaybreaks
\begin{alignat}{1}
 & \left\langle \frac{H(x)}{1-sE(yx)H(x)},\frac{H}{1-tE(z)H}\right\rangle =\left\langle \sum_{i=0}^{\infty}E(yx)^{i}H(x)^{i+1}s^{i},\sum_{j=0}^{\infty}E(z)^{j}H^{j+1}t^{j}\right\rangle \nonumber \\
 & \qquad\qquad\qquad\qquad=\sum_{i,j=0}^{\infty}\left\langle E(yx)^{i}H(x)^{i+1},E(z)^{j}H^{j+1}\right\rangle s^{i}t^{j}\nonumber \\
 & \qquad\qquad\qquad\qquad=\sum_{i,j=0}^{\infty}\left\langle (E(yx)^{i}H(x)^{i+1})[X+1],E(z)^{j}H^{j}\right\rangle s^{i}t^{j}\nonumber \\
 & \qquad\qquad\qquad\qquad=\sum_{i,j=0}^{\infty}\left\langle (E(yx)[X+1])^{i}(H(x)[X+1])^{i+1},E(z)^{j}H^{j}\right\rangle s^{i}t^{j}\nonumber \\
 & \qquad\qquad\qquad\qquad=\sum_{i,j=0}^{\infty}\left\langle \frac{(1+yx)^{i}}{(1-x)^{i+1}}E(yx)^{i}H(x)^{i+1},E(z)^{j}H^{j}\right\rangle s^{i}t^{j}\nonumber \\
 & \qquad\qquad\qquad\qquad=\sum_{i,j=0}^{\infty}\frac{(1+yx)^{i}}{(1-x)^{i+1}}\left\langle E(yx),E(z)^{j}H^{j}\right\rangle ^{i}\left\langle H(x),E(z)^{j}H^{j}\right\rangle ^{i+1}s^{i}t^{j}\nonumber \\
 & \qquad\qquad\qquad\qquad=\sum_{i,j=0}^{\infty}\frac{(1+yx)^{i(j+1)}}{(1-yzx)^{ij}}\frac{(1+zx)^{(i+1)j}}{(1-x)^{(i+1)(j+1)}}s^{i}t^{j};\label{e-lpkdessc2}
\end{alignat}
}here, we are using Lemma \ref{l-sppleth} for the third equality, and in the last three steps we apply Lemmas \ref{l-X+1}, \ref{l-pkdeshom}, and \ref{l-scalprodHE}, respectively. Combining (\ref{e-lpkdessc1}) and (\ref{e-lpkdessc2}) completes the proof.
\end{proof}

\subsection{Left peaks, inverse peaks, descents, and inverse descents}

Next, let us consider the joint distribution of $\lpk$, $\ipk$,
$\des$, and $\ides$ over $\mathfrak{S}_{n}$. Define
\begin{align*}
P_{n}^{(\lpk,\ipk,\des,\ides)}(y,z,s,t) & \coloneqq\sum_{\pi\in\mathfrak{S}_{n}}y^{\lpk(\pi)}z^{\ipk(\pi)+1}s^{\des(\pi)}t^{\ides(\pi)+1}.
\end{align*}

\begin{thm}
\label{t-pkdesilpkides}We have
\begin{multline*}
\!\frac{1}{(1-s)(1-t)}+\sum_{n=1}^{\infty}\frac{(1+ys)^{n}(1+zt)^{n+1}P_{n}^{(\lpk,\ipk,\des,\ides)}\left(\frac{(1+y)^{2}s}{(y+s)(1+ys)},\frac{(1+z)^{2}t}{(z+t)(1+zt)},\frac{y+s}{1+ys},\frac{z+t}{1+zt}\right)}{(1+z)(1-s)^{n+1}(1-t)^{n+1}}x^{n}\\
=\sum_{i,j=0}^{\infty}\left(\frac{1+yx}{1-yzx}\right)^{ij}\left(\frac{1+zx}{1-x}\right)^{(i+1)j}s^{i}t^{j}.
\end{multline*}
\end{thm}

\begin{proof}
The two sides of the above equation are obtained from evaluating the
scalar product
\[
\left\langle \frac{H}{1-sE(yx)H(x)},\frac{1}{1-tE(z)H}\right\rangle 
\]
in two different ways; the proof is similar to that of Theorem \ref{t-pkdes}
and so we omit the details.
\end{proof}

\subsection{Left peaks and inverse left peaks} \label{ss-lpkilpk}

We now consider specializations of $P_{n}^{(\lpk,\ilpk,\des,\ides)}(y,z,s,t)$
and $P_{n}^{(\lpk,\ipk,\des,\ides)}(y,z,s,t)$, beginning with the
two-sided left peak polynomials
\begin{align*}
P_{n}^{(\lpk,\ilpk)}(s,t) & \coloneqq P_{n}^{(\lpk,\ilpk,\des,\ides)}(s,t,1,1)=\sum_{\pi\in\mathfrak{S}_{n}}s^{\lpk(\pi)}t^{\ilpk(\pi)}.
\end{align*}

\begin{thm}
\label{t-lpk}We have \leqnomode 
\begin{multline*}
\tag{{a}}\frac{1}{(1-s)(1-t)}+\sum_{n=1}^{\infty}\frac{((1+s)(1+t))^{n}}{((1-s)(1-t))^{n+1}}P_{n}^{(\lpk,\ilpk)}\left(\frac{4s}{(1+s)^{2}},\frac{4t}{(1+t)^{2}}\right)x^{n}\\
=\sum_{i,j=0}^{\infty}\frac{(1+x)^{2ij+i+j}}{(1-x)^{2ij+i+j+1}}s^{i}t^{j}
\end{multline*}
and, for all $n\geq1$, we have
\begin{multline*}
\tag{{b}}\qquad\frac{((1+s)(1+t))^{n}}{((1-s)(1-t))^{n+1}}P_{n}^{(\lpk,\ilpk)}\left(\frac{4s}{(1+s)^{2}},\frac{4t}{(1+t)^{2}}\right)\\
=\sum_{i,j=0}^{\infty}\sum_{k=0}^{n}{2ij+i+j \choose k}{2ij+i+j+n-k \choose n-k}s^{i}t^{j}\qquad
\end{multline*}
and
\begin{alignat*}{1}
\tag{{c}}\quad\,\,\,\frac{((1+s)(1+t))^{n}}{((1-s)(1-t))^{n+1}}P_{n}^{(\lpk,\ilpk)}\left(\frac{4s}{(1+s)^{2}},\frac{4t}{(1+t)^{2}}\right) 
&=\sum_{k=0}^n e(n,k) \frac{B_{k}(s)B_{k}(t)}{(1-s)^{k+1}(1-t)^{k+1}}
\end{alignat*}
with $e(n,k)$ as defined in Section \ref{s-zsums}.
\end{thm}

\begin{proof}
Evaluating Theorem \ref{t-lpkdes} at $y=z=1$ yields part (a), whereas
(b) follows from (a) and
\[
\frac{(1+x)^{2ij+i+j}}{(1-x)^{2ij+i+j+1}}=\sum_{n=0}^{\infty}\sum_{k=0}^{n}{2ij+i+j \choose k}{2ij+i+j+n-k \choose n-k}x^{n}.
\]
To prove (c), first observe that Lemma \ref{l-psexp} (b) implies
\begin{alignat*}{1}
\left\langle \frac{H(x)}{1-sE(x)H(x)},\frac{H}{1-tEH}\right\rangle  & =\sum_{\lambda,\mu}\frac{1}{z_{\lambda}z_{\mu}}\frac{B_{o(\lambda)}(s)B_{o(\mu)}(t)}{(1-s)^{o(\lambda)+1}(1-t)^{o(\mu)+1}}\left\langle p_{\lambda},p_{\mu}\right\rangle x^{\left|\lambda\right|}\\
 & =\sum_{\lambda}\frac{1}{z_{\lambda}}\frac{B_{o(\lambda)}(s)B_{o(\lambda)}(t)}{(1-s)^{o(\lambda)+1}(1-t)^{o(\lambda)+1}}x^{\left|\lambda\right|},
\end{alignat*}
but this same scalar product is also given by
\begin{align*}
 & \left\langle \frac{H(x)}{1-sE(x)H(x)},\frac{H}{1-tEH}\right\rangle \\
 & \qquad\qquad=\frac{1}{(1-s)(1-t)}\sum_{n=1}^{\infty}\frac{((1+s)(1+t))^{n}}{((1-s)(1-t))^{n+1}}P_{n}^{(\lpk,\ilpk)}\left(\frac{4s}{(1+s)^{2}},\frac{4t}{(1+t)^{2}}\right)x^{n},
\end{align*}
which is obtained from evaluating (\ref{e-lpkdessc1}) at $y=z=1$.
Equating these two expressions, extracting coefficients of $x^{n}$, and applying Lemma \ref{l-cdef} (c) completes the proof.
\end{proof}

See Table \ref{tb-lpk} for the first several polynomials $P_{n}^{(\lpk,\ilpk)}(s,t)$.

\renewcommand{\arraystretch}{1.3}
\begin{table}[H]
\centering
\noindent \begin{raggedright}
{\small{}}%
\begin{tabular}{|>{\centering}m{0.1in}|>{\raggedright}p{6in}|}
\hline 
\multicolumn{1}{|c|}{{\small{}$n$}} & {\small{}$P_{n}^{(\lpk,\ilpk)}(s,t)$}\tabularnewline
\hline 
{\small{}$1$} & {\small{}$1$}\tabularnewline
\hline 
{\small{}$2$} & {\small{}$1+st$}\tabularnewline
\hline 
{\small{}$3$} & {\small{}$1+5st$}\tabularnewline
\hline 
{\small{}$4$} & {\small{}$1+(15s+3s^{2})t+(3s+2s^{2})t^{2}$}\tabularnewline
\hline 
{\small{}$5$} & {\small{}$1+(35s+23s^{2})t+(23s+38s^{2})t^{2}$}\tabularnewline
\hline 
{\small{}$6$} & {\small{}$1+(70s+100s^{2}+9s^{3})t+(100s+335s^{2}+44s^{3})t^{2}+(9s+44s^{2}+8s^{3})t^{3}$}\tabularnewline
\hline 
{\small{}$7$} & {\small{}$1+(126s+324s^{2}+93s^{3})t+(324s+1951s^{2}+836s^{3})t^{2}+(93s+836s^{2}+456s^{3})t^{3}$}\tabularnewline
\hline 
\end{tabular}{\small\par}
\par\end{raggedright}
\caption{\label{tb-lpk}Joint distribution of $\protect\lpk$ and $\protect\ilpk$
over $\mathfrak{S}_{n}$}
\end{table}

Next, we express these two-sided left peak polynomials as a sum involving products 
of the ordinary left peak polynomials
\[
P_{n}^{\lpk}(t)\coloneqq\sum_{\pi\in\mathfrak{S}_{n}}t^{\lpk(\pi)}.
\]

\begin{thm}
\label{t-lpkilpkprod}
For all $n\geq1$, we have
\begin{align*}
P_{n}^{(\lpk,\ilpk)}(s,t) 
& =\frac{1}{n!}\sum_{k=0}^n e(n,k)\left((1-s)(1-t)\right)^{\frac{n-k}{2}}P_{k}^{\lpk}(s)P_{k}^{\lpk}(t).
\end{align*}
\end{thm}

\begin{proof}
The proof follows in the same way as the proof of Theorem \ref{t-pkprod}
except that we use the formula
\begin{equation}
B_{n}(t)=(1+t)^{n}P_{n}^{\lpk}\left(\frac{4t}{(1+t)^{2}}\right)\label{e-Blpk}
\end{equation}
\cite[Proposition 4.15]{Petersen2007} in place of (\ref{e-Apk}).
The details are omitted.
\end{proof}

From Theorem \ref{t-lpkilpkprod} and Proposition \ref{p-cdef} (c), we can obtain a generating function for the two-sided left peak polynomials, analogous to Theorem \ref{t-pkipkgf} for the two-sided peak polynomials.

\begin{thm}
\label{t-lpkilpkgf}
We have
\begin{equation*}
\sum_{n=0}^\infty P_{n}^{(\lpk,\ilpk)}(s,t) u^n = \frac{1}{\sqrt{1-(1-s)(1-t)u^2}}\sum_{k=0}^\infty \frac{1}{k!} \left(\frac{L(\sqrt{(1-s)(1-t)}u)}{\sqrt{(1-s)(1-t)}}\right)^{\! k} P_{k}^{\lpk}(s)P_{k}^{\lpk}(t).
\end{equation*}
\end{thm}

Just as we derived from Theorem \ref{t-pkipkgf} the odd part of Stanley's formula \eqref{e-altalt} for doubly alternating permutations, we can derive the even part of \eqref{e-altalt} from Theorem \ref{t-lpkilpkgf}. First, we note that a permutation $\pi$ in $\mathfrak{S}_n$ has at most $n/2$ left peaks, and has exactly $n/2$ left peaks if and only if $n$ is even and $\pi$ is alternating. Then we have
\[
\lim_{s\rightarrow0}P_{k}^{\lpk}(s^{-2})s^{k}=\lim_{s\rightarrow0}\sum_{\pi\in\mathfrak{S}_{k}}s^{k-2\lpk(\pi)}=\begin{cases}
0, & \text{if }n\text{ is odd,}\\
E_{k}, & \text{if }n\text{ is even},
\end{cases}
\]
and similarly
\[
\lim_{s\rightarrow0}P_{n}^{(\lpk,\ilpk)}(s^{-2},t^{-2})s^{n}t^{n}=\begin{cases}
0, & \text{if }n\text{ is odd,}\\
\tilde{E}_{n}, & \text{if }n\text{ is even.}
\end{cases}
\]
Therefore, to obtain the even part of Stanley's formula, we take Theorem \ref{t-lpkilpkgf} and replace $s$ with $s^{-2}$, $t$ with $t^{-2}$, and $u$ with $stu$; then we take the limit as $s,t\to 0$, similarly to the computation for the odd part.

\subsection{Left peaks and inverse peaks}

We proceed to give analogous formulas for the polynomials
\begin{align*}
P_{n}^{(\lpk,\ipk)}(s,t) & \coloneqq P_{n}^{(\lpk,\ipk)}(s,t,1,1)=\sum_{\pi\in\mathfrak{S}_{n}}s^{\lpk(\pi)}t^{\ipk(\pi)+1}
\end{align*}
encoding the joint distribution of $\lpk$ and $\ipk$ over $\mathfrak{S}_n$.

\begin{thm}
\label{t-pkilpk}We have \leqnomode 
\begin{multline*}
\tag{{a}}\frac{1}{(1-s)(1-t)}+\frac{1}{2}\sum_{n=1}^{\infty}\frac{(1+s)^{n}(1+t)^{n+1}}{((1-s)(1-t))^{n+1}}P_{n}^{(\lpk,\ipk)}\left(\frac{4s}{(1+s)^{2}},\frac{4t}{(1+t)^{2}}\right)x^{n}\\
=\sum_{i,j=0}^{\infty}\left(\frac{1+x}{1-x}\right)^{2ij+j}s^{i}t^{j}
\end{multline*}
and, for all $n\geq1$, we have
\begin{multline*}
\tag{{b}}\qquad\frac{(1+s)^{n}(1+t)^{n+1}}{((1-s)(1-t))^{n+1}}P_{n}^{(\lpk,\ipk)}\left(\frac{4s}{(1+s)^{2}},\frac{4t}{(1+t)^{2}}\right)\\
=2\sum_{i,j=0}^{\infty}\sum_{k=0}^{n}{2ij+j \choose k}{2ij+j+n-k-1 \choose n-k}s^{i}t^{j}\qquad
\end{multline*}
and
\begin{multline*}
\tag{{c}}\qquad\frac{(1+s)^{n}(1+t)^{n+1}}{((1-s)(1-t))^{n+1}}P_{n}^{(\lpk,\ipk)}\left(\frac{4s}{(1+s)^{2}},\frac{4t}{(1+t)^{2}}\right)\\
=\frac{1}{n!}\sum_{k=0}^n
2^{k+1} d(n,k)
\frac{B_{k}(s)A_{k}(t)}{(1-s)^{k+1}(1-t)^{k+1}}.\qquad
\end{multline*}
\end{thm}

\begin{proof}
The proof of parts (a) and (b) is the same as that of Theorem \ref{t-lpk}
(a) and (b), except that we specialize Theorem \ref{t-pkdesilpkides}
as opposed to Theorem \ref{t-lpkdes}. To prove (c), notice that from
Lemma \ref{l-psexp} (a)\textendash (b) we have 
\begin{align}
\left\langle \frac{1}{1-sE(x)H(x)},\frac{H}{1-tEH}\right\rangle  & =\sum_{\lambda\text{ odd}}\sum_{\mu}\frac{2^{l(\lambda)}}{z_{\lambda}z_{\mu}}\frac{A_{l(\lambda)}(s)B_{o(\mu)}(t)}{(1-s)^{l(\lambda)+1}(1-t)^{o(\mu)+1}}\left\langle p_{\lambda},p_{\mu}\right\rangle x^{\left|\mu\right|}\nonumber \\
 & =\sum_{\lambda\text{ odd}}\frac{2^{l(\lambda)}}{z_{\lambda}}\frac{A_{l(\lambda)}(s)B_{l(\lambda)}(t)}{(1-s)^{l(\lambda)+1}(1-t)^{l(\lambda)+1}}x^{\left|\lambda\right|},\label{e-pkilpksc1}
\end{align}
and the same scalar product can be shown to be equal to
\begin{multline}
\left\langle \frac{1}{1-sE(x)H(x)},\frac{H}{1-tEH}\right\rangle =\frac{1}{(1-s)(1-t)}\\
+\frac{1}{2}\sum_{n=1}^{\infty}\frac{(1+s)^{n+1}(1+t)^{n}}{((1-s)(1-t))^{n+1}}P_{n}^{(\lpk,\ipk)}\left(\frac{4s}{(1+s)^{2}},\frac{4t}{(1+t)^{2}}\right)x^{n}\label{e-pkilpksc2}
\end{multline}
by using Lemma \ref{l-ribexp} (b)\textendash (c).
Equating (\ref{e-pkilpksc1}) and (\ref{e-pkilpksc2}), extracting coefficients of $x^{n}$, 
and then applying Lemma \ref{l-cdef} (b) completes the proof of (c).
\end{proof}

\renewcommand{\arraystretch}{1.3}
\begin{table}[H]
\centering
\noindent \begin{raggedright}
{\small{}}%
\begin{tabular}{|>{\centering}m{0.1in}|>{\raggedright}p{6in}|}
\hline 
\multicolumn{1}{|c|}{{\small{}$n$}} & {\small{}$P_{n}^{(\lpk,\ipk)}(s,t)$}\tabularnewline
\hline 
{\small{}$1$} & {\small{}$t$}\tabularnewline
\hline 
{\small{}$2$} & {\small{}$(1+s)t$}\tabularnewline
\hline 
{\small{}$3$} & {\small{}$(1+3s)t+2st^{2}$}\tabularnewline
\hline 
{\small{}$4$} & {\small{}$(1+6s+s^{2})t+(12s+4s^{2})t^{2}$}\tabularnewline
\hline 
{\small{}$5$} & {\small{}$(1+10s+5s^{2})t+(42s+46s^{2})t^{2}+(6s+10s^{2})t^{3}$}\tabularnewline
\hline 
{\small{}$6$} & {\small{}$(1+15s+15s^{2}+s^{3})t+(112s+272s^{2}+32s^{3})t^{2}+(52s+192s^{2}+28s^{3})t^{3}$}\tabularnewline
\hline 
{\small{}$7$} & {\small{}$(1+21s+35s^{2}+7s^{3})t+(252s+1136s^{2}+436s^{3})t^{2}+(252s+1776s^{2}+852s^{3})t^{3}+(18s+164s^{2}+90s^{3})t^{4}$}\tabularnewline
\hline 
\end{tabular}{\small\par}
\par\end{raggedright}
\caption{\label{tb-lpkipk}Joint distribution of $\protect\lpk$ and $\protect\ipk$
over $\mathfrak{S}_{n}$}
\end{table}

Table \ref{tb-lpkipk} displays the first several polynomials $P_{n}^{(\lpk,\ipk)}(s,t)$.
The following proposition, which is Corollary 11 of \cite{Troyka2022},
characterizes the coefficients of $s^{k}t$ in $P_{n}^{(\lpk,\ipk)}(s,t)$.
\begin{prop}
\label{p-lpkipk0}For any $n\geq1$ and $k\geq0$, the number of permutations
$\pi\in\mathfrak{S}_{n}$ with $\lpk(\pi)=k$ and $\ipk(\pi)=0$ is
equal to ${n \choose 2k}$.
\end{prop}

The next formula expresses the polynomial $P_{n}^{(\pk,\ilpk)}(s,t)$
as a sum involving products of the peak and left peak polynomials.
The proof follows in the same way as that of Theorem~\ref{t-pkprod},
except that we begin by substituting both (\ref{e-Apk}) and (\ref{e-Blpk})
into Theorem \ref{t-pkilpk} (c).
\begin{thm}
For all $n\geq1$, we have
\begin{align*}
P_{n}^{(\lpk,\ipk)}(s,t) 
& =\frac{1}{n!}\sum_{k=0}^n
d(n,k)
\left((1-s)(1-t)\right)^{\frac{n-k}{2}}P_{k}^{\lpk}(s)P_{k}^{\pk}(t).
\end{align*}
\end{thm}

\subsection{Left peaks and inverse descents}

Finally, define
\begin{align*}
P_{n}^{(\lpk,\ides)}(s,t) & \coloneqq P_{n}^{(\lpk,\ipk,\des,\ides)}(s,1,t,1)=\sum_{\pi\in\mathfrak{S}_{n}}s^{\lpk(\pi)}t^{\ides(\pi)+1}.
\end{align*}
We omit the proofs of the following theorems as they are very similar
to the proofs of analogous results presented earlier.

\begin{thm}
We have \leqnomode 
\begin{multline*}
\tag{{a}}\qquad\frac{1}{(1-s)(1-t)}+\sum_{n=1}^{\infty}\frac{(1+s)^{n}}{((1-s)(1-t))^{n+1}}P_{n}^{(\lpk,\ides)}\left(\frac{4s}{(1+s)^{2}},t\right)x^{n}\\
=\sum_{i,j=0}^{\infty}\frac{(1+x)^{i(j+1)}}{(1-x)^{(i+1)(j+1)}}s^{i}t^{j},\qquad
\end{multline*}
and, for all $n\geq1$, we have
\begin{multline*}
\tag{{b}}\qquad\frac{(1+s)^{n}}{((1-s)(1-t))^{n+1}}P_{n}^{(\lpk,\ides)}\left(\frac{4s}{(1+s)^{2}},t\right)\\
=\sum_{i,j=0}^{\infty}\sum_{k=0}^{n}{i(j+1) \choose k}{ij+i+j+n-k \choose n-k}s^{i}t^{j}\qquad
\end{multline*}
and
\begin{align*}
\tag{{c}}\quad\,\frac{(1+s)^{n}}{((1-s)(1-t))^{n+1}}P_{n}^{(\lpk,\ides)}\left(\frac{4s}{(1+s)^{2}},t\right) 
  & =\frac{1}{n!}\sum_{k,m=0}^n f(n,k,m)\frac{B_{k}(s)A_{m}(t)}{(1-s)^{k+1}(1-t)^{m+1}}
\end{align*}
with $f(n,k,m)$ as defined in Section \ref{s-zsums}.
\end{thm}

\begin{thm}
For all $n\geq1$, we have
\begin{align*}
P_{n}^{(\lpk,\ides)}(s,t) 
& =\frac{1}{n!}\sum_{k,m=0}^n f(n,k,m) (1-s)^{\frac{n-k}{2}}(1-t)^{n-m}P_{k}^{\lpk}(s)A_{m}(t).
\end{align*}
\end{thm}

See Table \ref{tb-desilpk} for the first several polynomials $P_{n}^{(\lpk,\ides)}(s,t)$.

\renewcommand{\arraystretch}{1.3}
\begin{table}[H]
\centering
\noindent \begin{raggedright}
{\small{}}%
\begin{tabular}{|>{\centering}m{0.1in}|>{\raggedright}p{6in}|}
\hline 
\multicolumn{1}{|c|}{{\small{}$n$}} & {\small{}$P_{n}^{(\lpk,\ides)}(s,t)$}\tabularnewline
\hline 
{\small{}$1$} & {\small{}$t$}\tabularnewline
\hline 
{\small{}$2$} & {\small{}$t+st^2$}\tabularnewline
\hline 
{\small{}$3$} & {\small{}$t+4st^2+st^{3}$}\tabularnewline
\hline 
{\small{}$4$} & {\small{}$t+(10s+s^{2})t^2+(7s+4s^{2})t^{3}+st^{4}$}\tabularnewline
\hline 
{\small{}$5$} & {\small{}$t+(20s+6s^{2})t^2+(27s+39s^{2})t^{3}+(10s+16s^{2})t^{4}+st^{5}$}\tabularnewline
\hline 
{\small{}$6$} & {\small{}$t+(35s+21s^{2}+s^{3})t^2+(77s+205s^{2}+20s^{3})t^{3}+(53s+213s^{2}+36s^{3})t^{4}+(13s+40s^{2}+4s^{3})t^{5}+st^{6}$}\tabularnewline
\hline 
{\small{}$7$} & {\small{}$t+(56s+56s^{2}+8s^{3})t^2+(182s+776s^{2}+233s^{3})t^{3}+(200s+1480s^{2}+736s^{3})t^{4}+(88s+719s^{2}+384s^{3})t^{5}+(16s+80s^{2}+24s^{3})t^{6}+st^{7}$}\tabularnewline
\hline 
\end{tabular}{\small\par}
\par\end{raggedright}
\caption{\label{tb-desilpk}Joint distribution of $\protect\lpk$ and $\protect\ides$
over $\mathfrak{S}_{n}$}
\end{table}

\section{\label{s-udr}Up-down runs and biruns}

We will now give analogous formulas for (mixed) two-sided distributions involving
the number of up-down runs, as well as a couple involving the number
of biruns.

\subsection{Up-down runs and inverse up-down runs}

Consider the two-sided up-down run polynomials
\[
P_{n}^{(\udr,\iudr)}(s,t)\coloneqq\sum_{\pi\in\mathfrak{S}_{n}}s^{\udr(\pi)}t^{\iudr(\pi)}.
\]

\begin{thm}
\label{t-udr}We have 
\begin{multline*}
\frac{1}{(1-s)(1-t)}+\frac{1}{4(1-s)^{2}(1-t)^{2}}\sum_{n=1}^{\infty}\frac{(1+s^{2})^{n}(1+t^{2})^{n}}{(1-s^{2})^{n-1}(1-t^{2})^{n-1}}P_{n}^{(\udr,\iudr)}\left(\frac{2s}{1+s^{2}},\frac{2t}{1+t^{2}}\right)x^{n}\\
=\sum_{i,j=0}^{\infty}\left(\frac{1+x}{1-x}\right)^{2ij}\left(1+s\left(\frac{1+x}{1-x}\right)^{j}+t\left(\frac{1+x}{1-x}\right)^{i}+st\frac{(1+x)^{i+j}}{(1-x)^{i+j+1}}\right)s^{2i}t^{2j}.
\end{multline*}
\end{thm}

The proof of Theorem \ref{t-udr} follows the same structure as our
proofs from Sections \ref{s-pkdes}\textendash \ref{s-lpk}: we compute
an appropriate scalar product in multiple ways and set them equal
to each other. We shall provide the full proof here, but will omit
the proofs for all remaining results in this section due to their
similarity with what has been presented earlier.
\begin{proof}
First, Lemma \ref{l-ribexp} (d) along with Foulkes's theorem (Theorem
\ref{t-foulkes}) leads to
\begin{alignat}{1}
 & \left\langle \frac{1+sH(x)}{1-s^{2}E(x)H(x)},\frac{1+tH}{1-t^{2}EH}\right\rangle =\frac{1}{(1-s)(1-t)}\nonumber \\
 & \qquad+\frac{1}{4(1-s)^{2}(1-t)^{2}}\sum_{n=1}^{\infty}\frac{(1+s^{2})^{n}(1+t^{2})^{n}}{(1-s^{2})^{n-1}(1-t^{2})^{n-1}}P_{n}^{(\udr,\iudr)}\left(\frac{2s}{1+s^{2}},\frac{2t}{1+t^{2}}\right)x^{n}.\label{e-udr0}
\end{alignat}
Next, using Lemmas \ref{l-pkdeshom} and \ref{l-scalprodHE}, we obtain {\allowdisplaybreaks
\begin{alignat}{1}
 & \left\langle \frac{1+sH(x)}{1-s^{2}E(x)H(x)},\frac{1}{1-t^{2}EH}\right\rangle =\sum_{i,j=0}^{\infty}\left\langle (1+sH(x))E(x)^{i}H(x)^{i},E^{j}H^{j}\right\rangle s^{2i}t^{2j}\nonumber \\
 & \qquad=\sum_{i,j=0}^{\infty}\left(\left\langle E(x)^{i}H(x)^{i},E^{j}H^{j}\right\rangle +s\left\langle E(x)^{i}H(x)^{i+1},E^{j}H^{j}\right\rangle \right)s^{2i}t^{2j}\nonumber \\
 & \qquad=\sum_{i,j=0}^{\infty}\left(\left\langle E(x),E^{j}H^{j}\right\rangle ^{i}\left\langle H(x),E^{j}H^{j}\right\rangle ^{i}+s\left\langle E(x),E^{j}H^{j}\right\rangle ^{i}\left\langle H(x),E^{j}H^{j}\right\rangle ^{i+1}\right)s^{2i}t^{2j}\nonumber \\
 & \qquad=\sum_{i,j=0}^{\infty}\left(\left(\frac{1+x}{1-x}\right)^{2ij}+s\left(\frac{1+x}{1-x}\right)^{2ij+j}\right)s^{2i}t^{2j}\nonumber \\
 & \qquad=\sum_{i,j=0}^{\infty}\left(\frac{1+x}{1-x}\right)^{2ij}\left(1+s\left(\frac{1+x}{1-x}\right)^{j}\right)s^{2i}t^{2j}.\label{e-udr1}
\end{alignat}
}Similarly, using Lemmas \ref{l-pkdeshom}\textendash \ref{l-X+1}, we obtain{\allowdisplaybreaks
\begin{alignat}{1}
 & \left\langle \frac{1+sH(x)}{1-s^{2}E(x)H(x)},\frac{tH}{1-t^{2}EH}\right\rangle =\sum_{i,j=0}^{\infty}\left\langle (1+sH(x))E(x)^{i}H(x)^{i},E^{j}H^{j+1}\right\rangle s^{2i}t^{2j+1}\nonumber \\
 & \qquad\qquad\qquad=\sum_{i,j=0}^{\infty}\left\langle ((1+sH(x))E(x)^{i}H(x)^{i})[X+1],E^{j}H^{j}\right\rangle s^{2i}t^{2j+1}\nonumber \\
 & \qquad\qquad\qquad=\sum_{i,j=0}^{\infty}\left\langle (E(x)^{i}H(x)^{i}+sE(x)^{i}H(x)^{i+1})[X+1],E^{j}H^{j}\right\rangle s^{2i}t^{2j+1}\nonumber \\
 & \qquad\qquad\qquad=\sum_{i,j=0}^{\infty}\left\langle (E(x)^{i}H(x)^{i}+sE(x)^{i}H(x)^{i+1})[X+1],E^{j}H^{j}\right\rangle s^{2i}t^{2j+1}\nonumber \\
 & \qquad\qquad\qquad=\sum_{i,j=0}^{\infty}\left\langle \left(\frac{1+x}{1-x}\right)^{i}E(x)^{i}H(x)^{i}+s\frac{(1+x)^{i}}{(1-x)^{i+1}}E(x)^{i}H(x)^{i+1},E^{j}H^{j}\right\rangle s^{2i}t^{2j+1}\nonumber \\
 & \qquad\qquad\qquad=\sum_{i,j=0}^{\infty}\Bigg(\left(\frac{1+x}{1-x}\right)^{i}\left\langle E(x),E^{j}H^{j}\right\rangle ^{i}\left\langle H(x),E^{j}H^{j}\right\rangle ^{i}\nonumber \\
 & \qquad\qquad\qquad\qquad\qquad\qquad\qquad+s\frac{(1+x)^{i}}{(1-x)^{i+1}}\left\langle E(x),E^{j}H^{j}\right\rangle ^{i}\left\langle H(x),E^{j}H^{j}\right\rangle ^{i+1}\Bigg)s^{2i}t^{2j+1}\nonumber \\
 & \qquad\qquad\qquad=\sum_{i,j=0}^{\infty}\left(\left(\frac{1+x}{1-x}\right)^{2ij+i}+s\frac{(1+x)^{2ij+i+j}}{(1-x)^{2ij+i+j+1}}\right)s^{2i}t^{2j+1}\nonumber \\
 & \qquad\qquad\qquad=\sum_{i,j=0}^{\infty}\left(\frac{1+x}{1-x}\right)^{2ij}\left(t\left(\frac{1+x}{1-x}\right)^{i}+st\frac{(1+x)^{i+j}}{(1-x)^{i+j+1}}\right)s^{2i}t^{2j}.\label{e-udr2}
\end{alignat}
}Summing (\ref{e-udr1}) and (\ref{e-udr2}) yields
\begin{multline*}
\left\langle \frac{1+sH(x)}{1-s^{2}E(x)H(x)},\frac{1+tH}{1-t^{2}EH}\right\rangle =\\
\sum_{i,j=0}^{\infty}\left(\frac{1+x}{1-x}\right)^{2ij}\left(1+s\left(\frac{1+x}{1-x}\right)^{j}+t\left(\frac{1+x}{1-x}\right)^{i}+st\frac{(1+x)^{i+j}}{(1-x)^{i+j+1}}\right)s^{2i}t^{2j};
\end{multline*}
comparing this with (\ref{e-udr0}) completes the proof.
\end{proof}

The first several polynomials $P_{n}^{(\udr,\iudr)}(s,t)$ are given
in Table \ref{tb-udr}. 
We note that the permutations in $\mathfrak{S}_{n}$
with $n$ up-down runs are precisely the alternating permutations
in $\mathfrak{S}_{n}$, so the coefficient of $s^{n}t^{n}$ of $P_{n}^{(\udr,\iudr)}(s,t)$
is the number of doubly alternating permutations in $\mathfrak{S}_{n}$.
In fact, Stanley's formula \eqref{e-altalt} for doubly alternating permutations
can be derived from Theorem \ref{t-udr}, although it is simpler
to work with our formulas for $(\pk,\ipk)$ and $(\lpk,\ilpk)$.

\renewcommand{\arraystretch}{1.3}
\begin{table}[H]
\centering
\noindent \begin{raggedright}
{\small{}}%
\begin{tabular}{|>{\centering}m{0.1in}|>{\raggedright}p{6in}|}
\hline 
\multicolumn{1}{|c|}{{\small{}$n$}} & {\small{}$P_{n}^{(\udr,\iudr)}(s,t)$}\tabularnewline
\hline 
{\small{}$1$} & {\small{}$st$}\tabularnewline
\hline 
{\small{}$2$} & {\small{}$st+s^{2}t^{2}$}\tabularnewline
\hline 
{\small{}$3$} & {\small{}$st+(2s^{2}+s^{3})t^{2}+(s^{2}+s^{3})t^{3}$}\tabularnewline
\hline 
{\small{}$4$} & {\small{}$st+(3s^{2}+3s^{3}+s^{4})t^{2}+(3s^{2}+6s^{3}+2s^{4})t^{3}+(s^{2}+2s^{3}+2s^{4})t^{4}$}\tabularnewline
\hline 
{\small{}$5$} & {\small{}$st+(4s^{2}+6s^{3}+4s^{4}+s^{5})t^{2}+(6s^{2}+19s^{3}+13s^{4}+5s^{5})t^{3}+(4s^{2}+13s^{3}+21s^{4}+7s^{5})t^{4}+(s^{2}+5s^{3}+7s^{4}+3s^{5})t^{5}$}\tabularnewline
\hline 
{\small{}$6$} & {\small{}$st+(5s^{2}+10s^{3}+10s^{4}+5s^{5}+s^{6})t^{2}+(10s^{2}+45s^{3}+47s^{4}+38s^{5}+8s^{6})t^{3}+(10s^{2}+47s^{3}+109s^{4}+78s^{5}+24s^{6})t^{4}+(5s^{2}+38s^{3}+78s^{4}+70s^{5}+20s^{6})t^{5}+(s^{2}+8s^{3}+24s^{4}+20s^{5}+8s^{6})t^{6}$}\tabularnewline
\hline 
\end{tabular}{\small\par}
\par\end{raggedright}
\caption{\label{tb-udr}Joint distribution of $\protect\udr$ and $\protect\iudr$
over $\mathfrak{S}_{n}$}
\end{table}

To invert Theorem \ref{t-udr} and other formulas in this section,
we will need to solve $s=2u/(1+u^{2})$ for $u$, and the solution
is given by $u=s^{-1}(1-\sqrt{1-s^{2}})=(s/2)C(s^{2}/4)$ where $C(x)$
is the Catalan number generating function. Hence, Theorem \ref{t-udr}
can be written as
\begin{multline*}
\frac{1}{(1-u)(1-v)}+\frac{1}{4(1-u)^{2}(1-v)^{2}}\sum_{n=1}^{\infty}\frac{(1+u^{2})^{n}(1+v^{2})^{n}}{(1-u^{2})^{n-1}(1-v^{2})^{n-1}}P_{n}^{(\udr,\iudr)}(s,t)x^{n}\\
=\sum_{i,j=0}^{\infty}\left(\frac{1+x}{1-x}\right)^{2ij}\left(1+u\left(\frac{1+x}{1-x}\right)^{j}+v\left(\frac{1+x}{1-x}\right)^{i}+uv\frac{(1+x)^{i+j}}{(1-x)^{i+j+1}}\right)u^{2i}v^{2j}
\end{multline*}
where $u=s^{-1}(1-\sqrt{1-s^{2}})=(s/2)C(s^{2}/4)$ and $v=t^{-1}(1-\sqrt{1-t^{2}})=(t/2)C(t^{2}/4)$.

\subsection{Up-down runs, inverse peaks, and inverse descents}

Next, let us study the statistic $(\udr,\ipk,\ides)$ and its specializations
$(\udr,\ipk)$ and $(\udr,\ides)$. Define 
\begin{align*}
P_{n}^{(\udr,\ipk,\ides)}(s,y,t) & \coloneqq\sum_{\pi\in\mathfrak{S}_{n}}s^{\udr(\pi)}y^{\ipk(\pi)+1}t^{\ides(\pi)+1},\\
P_{n}^{(\udr,\ipk)}(s,t) & \coloneqq P_{n}^{(\udr,\ipk,\ides)}(s,t,1)=\sum_{\pi\in\mathfrak{S}_{n}}s^{\udr(\pi)}t^{\ipk(\pi)+1},\text{ and}\\
P_{n}^{(\udr,\ides)}(s,t) & \coloneqq P_{n}^{(\udr,\ipk,\ides)}(s,1,t)=\sum_{\pi\in\mathfrak{S}_{n}}s^{\udr(\pi)}t^{\ides(\pi)+1}.
\end{align*}

\begin{thm}
\label{t-udripkides}We have \leqnomode 
\begin{multline*}
\tag{{a}}\frac{1}{(1-s)(1-t)}+\sum_{n=1}^{\infty}\frac{(1+s^{2})^{n}}{(1-s^{2})^{n-1}}\left(\frac{1+yt}{1-t}\right)^{n+1}\frac{P_{n}^{(\udr,\ipk,\ides)}\left(\frac{2s}{1+s^{2}},\frac{(1+y)^{2}t}{(y+t)(1+yt)},\frac{y+t}{1+yt}\right)}{2(1-s)^{2}(1+y)}x^{n}\\
=\sum_{i,j=0}^{\infty}\left(\frac{(1+x)(1+yx)}{(1-yx)(1-x)}\right)^{ij}\left(1+s\left(\frac{1+yx}{1-x}\right)^{j}\right)s^{2i}t^{j}
\end{multline*}
and, for all $n\geq1$, we have
\begin{alignat*}{1}
\tag{{b}}\qquad & \frac{(1+s^{2})^{n}}{(1-s^{2})^{n-1}}\left(\frac{1+yt}{1-t}\right)^{n+1}\frac{P_{n}^{(\udr,\ipk,\ides)}\left(\frac{2s}{1+s^{2}},\frac{(1+y)^{2}t}{(y+t)(1+yt)},\frac{y+t}{1+yt}\right)}{2(1-s)^{2}(1+y)}\\
 & \qquad\qquad\qquad\qquad\qquad\qquad=\sum_{\substack{\lambda\vdash n\\
\mathrm{odd}
}
}\frac{2^{l(\lambda)}}{z_{\lambda}}\frac{A_{l(\lambda)}(s^{2})A_{l(\lambda)}(t)}{(1-s^{2})^{l(\lambda)+1}(1-t)^{l(\lambda)+1}}\prod_{k=1}^{l(\lambda)}(1-(-y)^{\lambda_{k}})\\
 & \qquad\qquad\qquad\qquad\qquad\qquad\quad+s\sum_{\lambda\vdash n}\frac{1}{z_{\lambda}}\frac{B_{o(\lambda)}(s^{2})A_{l(\lambda)}(t)}{(1-s^{2})^{o(\lambda)+1}(1-t)^{l(\lambda)+1}}\prod_{k=1}^{l(\lambda)}(1-(-y)^{\lambda_{k}}).
\end{alignat*}
\end{thm}

As before, we set $y=1$ to specialize to $(\udr,\ipk)$, immediately
arriving at parts (a) and (b) of the following theorem. Part (c) is
proven similarly to Theorem \ref{t-pkprod}, except that we also use
the formula 
\begin{equation*}
2^{n}A_{n}(t^{2})+tB_{n}(t^{2})=\frac{(1+t)^{2}(1+t^{2})^{n}}{2}P_{n}^{\udr}\left(\frac{2t}{1+t^{2}}\right)
\end{equation*}
\cite[Section 6.3]{Gessel2020} where
\[
P_{n}^{\udr}(t)\coloneqq\sum_{\pi\in\mathfrak{S}_{n}}t^{\udr(\pi)}.
\]

\begin{thm}
\label{t-udripk}We have \leqnomode 
\begin{multline*}
\tag{{a}}\frac{1}{(1-s)(1-t)}+\frac{1}{4(1-s)^{2}}\sum_{n=1}^{\infty}\frac{(1+s^{2})^{n}}{(1-s^{2})^{n-1}}\left(\frac{1+t}{1-t}\right)^{n+1}P_{n}^{(\udr,\ipk)}\left(\frac{2s}{1+s^{2}},\frac{4t}{(1+t)^{2}}\right)x^{n}\\
=\sum_{i,j=0}^{\infty}\left(\frac{1+x}{1-x}\right)^{2ij}\left(1+s\left(\frac{1+x}{1-x}\right)^{j}\right)s^{2i}t^{j}
\end{multline*}
and, for all $n\geq1$, we have 
\begin{multline*}
\tag{{b}}\qquad\frac{1}{4(1-s)^{2}}\frac{(1+s^{2})^{n}}{(1-s^{2})^{n-1}}\left(\frac{1+t}{1-t}\right)^{n+1}P_{n}^{(\udr,\ipk)}\left(\frac{2s}{1+s^{2}},\frac{4t}{(1+t)^{2}}\right)\\
=\frac{1}{n!}\sum_{k=0}^n 2^{k}d(n,k)\frac{(2^{k}A_{k}(s^{2})+sB_{k}(s^{2}))A_{k}(t)}{(1-s^{2})^{k+1}(1-t)^{k+1}}\qquad
\end{multline*}
and 
\begin{align*}
\tag{{c}}\qquad P_{n}^{(\udr,\ipk)}(s,t) & =\frac{1}{n!}\sum_{k=0}^n d(n,k)\left((1-s^{2})(1-t)\right)^{\frac{n-k}{2}}P_{k}^{\udr}(s)P_{k}^{\pk}(t).
\end{align*}
\end{thm}

\renewcommand{\arraystretch}{1.3}
\begin{table}[H]
\centering
\noindent \begin{raggedright}
{\small{}}%
\begin{tabular}{|>{\centering}m{0.1in}|>{\raggedright}p{6in}|}
\hline 
\multicolumn{1}{|c|}{{\small{}$n$}} & {\small{}$P_{n}^{(\udr,\ipk)}(s,t)$}\tabularnewline
\hline 
{\small{}$1$} & {\small{}$st$}\tabularnewline
\hline 
{\small{}$2$} & {\small{}$(s+s^{2})t$}\tabularnewline
\hline 
{\small{}$3$} & {\small{}$(s+2s^{2}+s^{3})t+(s^{2}+s^{3})t^{2}$}\tabularnewline
\hline 
{\small{}$4$} & {\small{}$(s+3s^{2}+3s^{3}+s^{4})t+(4s^{2}+8s^{3}+4s^{4})t^{2}$}\tabularnewline
\hline 
{\small{}$5$} & {\small{}$(s+4s^{2}+6s^{3}+4s^{4}+s^{5})t+(10s^{2}+32s^{3}+34s^{4}+12s^{5})t^{2}+(s^{2}+5s^{3}+7s^{4}+3s^{5})t^{3}$}\tabularnewline
\hline 
{\small{}$6$} & {\small{}$(s+5s^{2}+10s^{3}+10s^{4}+5s^{5}+s^{6})t+(20s^{2}+92s^{3}+156s^{4}+116s^{5}+32s^{6})t^{2}+(6s^{2}+46s^{3}+102s^{4}+90s^{5}+28s^{6})t^{3}$}\tabularnewline
\hline 
{\small{}$7$} & {\small{}$(s+6s^{2}+15s^{2}+20s^{4}+15s^{5}+6s^{6}+s^{7})t+(35s^{2}+217s^{3}+522s^{4}+614s^{5}+355s^{6}+81s^{7})t^{2}+(21s^{2}+231s^{3}+738s^{4}+1038s^{5}+681s^{6}+171s^{7})t^{3}+(s^{2}+17s^{3}+64s^{4}+100s^{5}+71s^{6}+19s^{7})t^{4}$}\tabularnewline
\hline 
\end{tabular}{\small\par}
\par\end{raggedright}
\caption{\label{tb-udripk}Joint distribution of $\protect\udr$ and $\protect\ipk$
over $\mathfrak{S}_{n}$}
\end{table}

The first several polynomials $P_{n}^{(\udr,\ipk)}(s,t)$ are displayed
in Table \ref{tb-udripk}. The coefficients of $s^{k}t$
in $P_{n}^{(\udr,\ipk)}(s,t)$ are binomial coefficients, just like
the coefficients of $st^{k}$ in $P_{n}^{(\pk,\ides)}(s,t)$ from
Table \ref{tb-desipk}.
\begin{prop}
\label{p-udripk0}For any $n\geq1$ and $k\geq0$, the number of permutations
$\pi\in\mathfrak{S}_{n}$ with $\udr(\pi)=k$ and $\ipk(\pi)=0$ is
equal to ${n-1 \choose k-1}$.
\end{prop}

Unlike the analogous results given earlier, Proposition \ref{p-udripk0}
was not proven in \cite{Troyka2022}, so we shall supply a proof here.
\begin{proof}
Let us define the \textit{up-down composition} of a permutation $\pi$, 
denoted $\udComp(\pi)$, in the following way: if $u_{1},u_{2},\dots,u_{k}$
are the lengths of the up-down runs of $\pi$ in the order that they
appear, then $\udComp(\pi)\coloneqq(u_{1},u_{2}-1,\dots,u_{k}-1)$.
For example, if $\pi=312872569$, then $\udComp(\pi)=(1,1,2,2,3)$.
Note that if $\pi\in\mathfrak{S}_{n}$, then $\udComp(\pi)$ is a
composition of $n$. It is not hard to verify that the descent composition
determines the up-down composition and vice versa.

According to \cite[Theorem 5]{Troyka2022}, for any composition $L\vDash n$,
there exists exactly one permutation $\pi\in\mathfrak{S}_{n}$ with
descent composition $L$ such that $\ipk(\pi)=0$. In light of this
fact and Proposition \ref{p-pkides0}, it suffices to construct a
bijection between compositions $L\vDash n$ with $k$ parts and compositions
$K\vDash n$ satisfying $\udr(K)=k$. This bijection is obtained by
mapping $L$ to the descent composition $K$ of any permutation with
up-down composition $L$. For example, the composition $L=(1,1,2,2,3)$
is mapped to $K=(1,3,1,4)$, which is the descent composition of the
permutation $\pi$ from above.
\end{proof}
Setting $y=0$ in Theorem \ref{t-udripkides} yields the analogous
result for $(\udr,\ides)$.

\begin{thm}
\label{t-udrides}We have \leqnomode 
\begin{multline*}
\tag{{a}}\qquad\frac{1}{(1-s)(1-t)}+\frac{1}{2(1-s)^{2}}\sum_{n=1}^{\infty}\frac{(1+s^{2})^{n}}{(1-s^{2})^{n-1}(1-t)^{n+1}}P_{n}^{(\udr,\ides)}\left(\frac{2s}{1+s^{2}},t\right)x^{n}\\
=\sum_{i,j=0}^{\infty}\left(\frac{1+x}{1-x}\right)^{ij}\left(1+\frac{s}{(1-x)^{j}}\right)s^{2i}t^{j}
\end{multline*}
and, for all $n\geq1$, we have 
\begin{multline*}
\tag{{b}} \frac{(1+s^{2})^{n}}{2(1-s)^{2}(1-s^{2})^{n-1}(1-t)^{n+1}}P_{n}^{(\udr,\ides)}\left(\frac{2s}{1+s^{2}},t\right)x^{n}\\
=\frac{1}{n!}\sum_{k=0}^n 2^k d(n,k)\frac{A_{k}(s^{2})A_{k}(t)}{(1-s^{2})^{k+1}(1-t)^{k+1}}
+\frac{s}{n!}\sum_{k,m=0}^n f(n,k,m)\frac{B_{k}(s^{2})A_{m}(t)}{(1-s^{2})^{k+1}(1-t)^{m+1}}.
\end{multline*}
\end{thm}

Table \ref{tb-udrides} contains the first several polynomials $P_{n}^{(\udr,\ides)}(s,t)$.

\renewcommand{\arraystretch}{1.3}
\begin{table}[H]
\centering
\noindent \begin{raggedright}
{\small{}}%
\begin{tabular}{|>{\centering}m{0.1in}|>{\raggedright}p{6in}|}
\hline 
\multicolumn{1}{|c|}{{\small{}$n$}} & {\small{}$P_{n}^{(\udr,\ides)}(s,t)$}\tabularnewline
\hline 
{\small{}$1$} & {\small{}$st$}\tabularnewline
\hline 
{\small{}$2$} & {\small{}$st+s^{2}t^{2}$}\tabularnewline
\hline 
{\small{}$3$} & {\small{}$st+(2s^{2}+2s^{3})t^{2}+s^{2}t^{3}$}\tabularnewline
\hline 
{\small{}$4$} & {\small{}$st+(3s^{2}+7s^{3}+s^{4})t^{2}+(3s^{2}+4s^{3}+4s^{4})t^{3}+s^{2}t^{4}$}\tabularnewline
\hline 
{\small{}$5$} & {\small{}$st+(4s^{2}+16s^{3}+4s^{4}+2s^{5})t^{2}+(6s^{2}+21s^{3}+27s^{4}+12s^{5})t^{3}+(4s^{2}+6s^{3}+14s^{4}+2s^{5})t^{4}+s^{2}t^{5}$}\tabularnewline
\hline 
{\small{}$6$} & {\small{}$st+(5s^{2}+30s^{3}+10s^{4}+11s^{5}+s^{6})t^{2}+(10s^{2}+67s^{3}+101s^{4}+104s^{5}+20s^{6})t^{3}+(10s^{2}+43s^{3}+125s^{4}+88s^{5}+36s^{6})t^{4}+(5s^{2}+8s^{3}+32s^{4}+8s^{5}+4s^{6})t^{5}+s^{2}t^{6}$}\tabularnewline
\hline 
\end{tabular}{\small\par}
\par\end{raggedright}
\caption{\label{tb-udrides}Joint distribution of $\protect\udr$ and $\protect\ides$
over $\mathfrak{S}_{n}$}
\end{table}

\subsection{Up-down runs, inverse left peaks, and inverse descents}

Define 
\begin{align*}
P_{n}^{(\udr,\ilpk,\ides)}(s,y,t) & \coloneqq\sum_{\pi\in\mathfrak{S}_{n}}s^{\udr(\pi)}y^{\ilpk(\pi)}t^{\ides(\pi)}\quad\text{and}\\
P_{n}^{(\udr,\ilpk)}(s,t) & \coloneqq P_{n}^{(\udr,\ilpk,\ides)}(s,t,1)=\sum_{\pi\in\mathfrak{S}_{n}}s^{\udr(\pi)}t^{\ilpk(\pi)}.
\end{align*}

\begin{thm}
\label{t-udrilpkides}We have 
\begin{multline*}
\frac{1}{(1-s)(1-t)}+\frac{1}{2(1-s)^{2}}\sum_{n=1}^{\infty}\frac{(1+s^{2})^{n}(1+yt)^{n}P_{n}^{(\udr,\ilpk,\ides)}\left(\frac{2s}{1+s^{2}},\frac{(1+y)^{2}t}{(y+t)(1+yt)},\frac{y+t}{1+yt}\right)}{(1-s^{2})^{n-1}(1-t)^{n+1}}x^{n}\\
=\sum_{i,j=0}^{\infty}\left(\frac{1+x}{1-x}\right)^{i(j+1)}\left(\frac{1+yx}{1-yx}\right)^{ij}\left(1+s\frac{(1+yx)^{j}}{(1-x)^{j+1}}\right)s^{2i}t^{j}.
\end{multline*}
\end{thm}

Setting $y=1$ in Theorem \ref{t-udrilpkides} yields part (a) of
the following theorem. Part (b) is obtained by computing the scalar
product 
\[
\left\langle \frac{1+sH(x)}{1-sE(x)H(x)},\frac{H}{1-tEH}\right\rangle 
\]
using Lemmas \ref{l-ribexp} (c)\textendash (d) and \ref{l-psexp}
(b)\textendash (c). Part (c) is obtained from (b) in a way similar
to the proof of Theorem \ref{t-lpkilpkprod}, making use of Equation
(\ref{e-Blpk}).
\begin{thm}
We have \leqnomode
\begin{multline*}
\tag{{a}}\frac{1}{(1-s)(1-t)}+\frac{1}{2(1-s)^{2}}\sum_{n=1}^{\infty}\frac{(1+s^{2})^{n}(1+t)^{n}}{(1-s^{2})^{n-1}(1-t)^{n+1}}P_{n}^{(\udr,\ilpk)}\left(\frac{2s}{1+s^{2}},\frac{4t}{(1+t)^{2}}\right)x^{n}\\
=\sum_{i,j=0}^{\infty}\left(\frac{1+x}{1-x}\right)^{i(2j+1)}\left(1+s\frac{(1+x)^{j}}{(1-x)^{j+1}}\right)s^{2i}t^{j}
\end{multline*}
and, for all $n\geq1$, we have 
\begin{multline*}
\tag{{b}}\frac{(1+s^{2})^{n}(1+t)^{n}}{2(1-s)^{2}(1-s^{2})^{n-1}(1-t)^{n+1}}
  P_{n}^{(\udr,\ilpk)}\left(\frac{2s}{1+s^{2}},\frac{4t}{(1+t)^{2}}\right)\\
=\frac{1}{n!}\sum_{k=0}^n 2^k d(n,k)
 \frac{A_{k}(s^{2})B_{k}(t)}{(1-s^{2})^{k+1}(1-t)^{k+1}}
  +\frac{s}{n!}\sum_{k=0}^n e(n,k) \frac{B_{k}(s^{2})B_{k}(t)}{(1-s^{2})^{k+1}(1-t)^{k+1}}
\end{multline*}
and 
\begin{align*}
\tag{{c}} \qquad\qquad(1+s)^{2}(1+s^{2})^{n}P_{n}^{(\udr,\ilpk)}&\left(\frac{2s}{1+s^{2}},t\right)\\
  & \mkern -30mu =\frac{1}{n!}\sum_{k=0}^n
  2^{k+1} d(n,k) \left((1-s^{2})(1-t)^{1/2}\right)^{n-k}A_{k}(s^{2})P_{k}^{\lpk}(t)\\
 & \quad\,\, +\frac{2s}{n!}\sum_{k=0}^n e(n,k)\left((1-s^{2})(1-t)^{1/2}\right)^{n-k}B_{k}(s^{2})P_{k}^{\lpk}(t).
\end{align*}
\end{thm}

See Table \ref{tb-udrilpk} for the first several polynomials $P_{n}^{(\udr,\ilpk)}(s,t)$.

\renewcommand{\arraystretch}{1.3}
\begin{table}[H]
\centering
\noindent \begin{raggedright}
{\small{}}%
\begin{tabular}{|>{\centering}m{0.1in}|>{\raggedright}p{6in}|}
\hline 
\multicolumn{1}{|c|}{{\small{}$n$}} & {\small{}$P_{n}^{(\udr,\ilpk)}(s,t)$}\tabularnewline
\hline 
{\small{}$1$} & {\small{}$s$}\tabularnewline
\hline 
{\small{}$2$} & {\small{}$s+s^{2}t$}\tabularnewline
\hline 
{\small{}$3$} & {\small{}$s+(3s^{2}+2s^{3})t$}\tabularnewline
\hline 
{\small{}$4$} & {\small{}$s+(6s^{2}+9s^{3}+3s^{4})t+(s^{2}+2s^{3}+2s^{4})t^{2}$}\tabularnewline
\hline 
{\small{}$5$} & {\small{}$s+(10s^{2}+25s^{3}+17s^{4}+6s^{5})t+(5s^{2}+18s^{3}+28s^{4}+10s^{5})t^{2}$}\tabularnewline
\hline 
{\small{}$6$} & {\small{}$s+(15s^{2}+55s^{3}+57s^{4}+43s^{5}+9s^{6})t+(15s^{2}+85s^{3}+187s^{4}+148s^{5}+44s^{6})t^{2}+(s^{2}+8s^{3}+24s^{4}+20s^{5}+8s^{6})t^{3}$}\tabularnewline
\hline 
{\small{}$7$} & {\small{}$s+(21s^{2}+105s^{3}+147s^{4}+177s^{5}+75s^{6}+18s^{7})t+(35s^{2}+289s^{3}+847s^{4}+1104s^{5}+672s^{6}+164s^{7})t^{2}+(7s^{2}+86s^{3}+350s^{4}+486s^{5}+366s^{6}+90s^{7})t^{3}$}\tabularnewline
\hline 
\end{tabular}{\small\par}
\par\end{raggedright}
\caption{\label{tb-udrilpk}Joint distribution of $\protect\udr$ and $\protect\ilpk$
over $\mathfrak{S}_{n}$}
\end{table}

\subsection{Biruns}

Recall that a birun is a maximal monotone consecutive subsequence,
whereas an up-down run is either a birun or an initial descent. In
\cite{Gessel2020}, the authors gave the formula 
\begin{align*}
\frac{2+tH(x)+tE(x)}{1-t^{2}E(x)H(x)} & =\frac{2}{1-t}+\frac{2t}{(1-t)^{2}}xh_{1}+\frac{(1+t)^{3}}{2(1-t)}\sum_{n=2}^{\infty}\sum_{L\vDash n}\frac{(1+t^{2})^{n-1}}{(1-t^{2})^{n}}\left(\frac{2t}{1+t^{2}}\right)^{\br(L)}x^{n}r_{L}
\end{align*}
(cf.\ Lemma \ref{l-ribexp}), which allows us to produce
formulas for (mixed) two-sided distributions involving the number
of biruns. For example, computing the scalar product 
\begin{equation}
\left\langle \frac{2+sH(x)+sE(x)}{1-s^{2}E(x)H(x)},\frac{2+tH+tE}{1-t^{2}EH}\right\rangle ,\label{e-spbr}
\end{equation}
would yield a formula for the two-sided distribution of $\br$, but
we chose not to derive this formula as it would be complicated to
write down. Looking back at the formula in Theorem~\ref{t-udr} for
the two-sided distribution of $\udr$, we see that the right-hand
side is a summation whose summands are each a sum involving four terms,
which is because the numerators in the scalar product
\[
\left\langle \frac{1+sH(x)}{1-s^{2}E(x)H(x)},\frac{1+tH}{1-t^{2}EH}\right\rangle 
\]
that we sought to compute has two terms each. The numerators in the
scalar product (\ref{e-spbr}) contain three terms each, which will
lead to nine terms in the formula for $(\br,\ibr)$ as opposed to
four.

On the other hand, formulas for the polynomials
\[
P_{n}^{(\br,\ipk)}(s,t)\coloneqq\sum_{\pi\in\mathfrak{S}_{n}}s^{\br(\pi)}t^{\ipk(\pi)+1}\quad\text{and}\quad P_{n}^{(\br,\ides)}(s,t)\coloneqq\sum_{\pi\in\mathfrak{S}_{n}}s^{\br(\pi)}t^{\ides(\pi)+1}
\]
have fewer such terms and so we present them below. 
\begin{thm}
\label{t-brides}We have \leqnomode
\begin{multline*}
\tag{{a}}\!\!\frac{1}{(1-s)(1-t)}+\frac{2stx}{(1-s)^{2}(1-t)^{2}}+\frac{(1+s)^{3}}{8(1-s)}\sum_{n=2}^{\infty}\frac{(1+s^{2})^{n-1}(1+t)^{n+1}P_{n}^{(\br,\ipk)}\left(\frac{2s}{1+s^{2}},\frac{4t}{(1+t)^{2}}\right)}{(1-s^{2})^{n}(1-t)^{n+1}}x^{n}\\
=\sum_{i,j=0}^{\infty}\left(\frac{1+x}{1-x}\right)^{2ij}\left(1+s\left(\frac{1+x}{1-x}\right)^{j}\right)s^{2i}t^{j}
\end{multline*}
and
\begin{multline*}
\tag{{b}}\frac{1}{(1-s)(1-t)}+\frac{stx}{(1-s)^{2}(1-t)^{2}}+\frac{(1+s)^{3}}{4(1-s)}\sum_{n=2}^{\infty}\frac{(1+s^{2})^{n-1}P_{n}^{(\br,\ides)}\left(\frac{2s}{1+s^{2}},t\right)}{(1-s^{2})^{n}(1-t)^{n+1}}x^{n}\\
=\frac{1}{2}\sum_{i,j=0}^{\infty}\left(\frac{1+x}{1-x}\right)^{ij}\left(2+\frac{s}{(1-x)^{j}}+s(1+x)^{j}\right)s^{2i}t^{j}.
\end{multline*}
\end{thm}

\renewcommand{\arraystretch}{1.3}
\begin{table}[H]
\centering
\noindent \begin{raggedright}
{\small{}}%
\begin{tabular}{|>{\centering}m{0.1in}|>{\raggedright}p{6in}|}
\hline 
\multicolumn{1}{|c|}{{\small{}$n$}} & {\small{}$P_{n}^{(\br,\ipk)}(s,t)$}\tabularnewline
\hline 
{\small{}$1$} & {\small{}$st$}\tabularnewline
\hline 
{\small{}$2$} & {\small{}$2st$}\tabularnewline
\hline 
{\small{}$3$} & {\small{}$(2s+2s^{2})t+2s^{2}t^{2}$}\tabularnewline
\hline 
{\small{}$4$} & {\small{}$(2s+4s^{2}+2s^{3})t+(8s^{2}+8s^{3})t^{2}$}\tabularnewline
\hline 
{\small{}$5$} & {\small{}$(2s+6s^{2}+6s^{3}+2s^{4})t+(20s^{2}+44s^{3}+24s^{4})t^{2}+(2s^{2}+8s^{3}+6s^{4})t^{3}$}\tabularnewline
\hline 
{\small{}$6$} & {\small{}$(2s+8s^{2}+12s^{3}+8s^{4}+2s^{5})t+(40s^{2}+144s^{3}+168s^{4}+64s^{5})t^{2}+(12s^{2}+80s^{3}+124s^{4}+56s^{5})t^{3}$}\tabularnewline
\hline 
{\small{}$7$} & {\small{}$(2s+10s^{2}+20s^{3}+20s^{4}+10s^{5}+2s^{6})t+(70s^{2}+364s^{3}+680s^{4}+548s^{5}+162s^{6})t^{2}+(42s^{2}+420s^{3}+1056s^{4}+1020s^{5}+342s^{6})t^{3}+(2s^{2}+32s^{3}+96s^{4}+104s^{5}+38s^{6})t^{4}$}\tabularnewline
\hline 
\end{tabular}{\small\par}
\par\end{raggedright}
\caption{\label{tb-bripk}Joint distribution of $\protect\br$ and $\protect\ipk$
over $\mathfrak{S}_{n}$}
\end{table}

\renewcommand{\arraystretch}{1.3}
\begin{table}[H]
\centering
\noindent \begin{raggedright}
{\small{}}%
\begin{tabular}{|>{\centering}m{0.1in}|>{\raggedright}p{6in}|}
\hline 
\multicolumn{1}{|c|}{{\small{}$n$}} & {\small{}$P_{n}^{(\br,\ides)}(s,t)$}\tabularnewline
\hline 
{\small{}$1$} & {\small{}$ts$}\tabularnewline
\hline 
{\small{}$2$} & {\small{}$(t+t^{2})s$}\tabularnewline
\hline 
{\small{}$3$} & {\small{}$(t+t^{3})s+4t^{2}s^{2}$}\tabularnewline
\hline 
{\small{}$4$} & {\small{}$(t+t^{4})s+(6t^{2}+6t^{3})s^{2}+(5t^{2}+5t^{3})s^{3}$}\tabularnewline
\hline 
{\small{}$5$} & {\small{}$(t+t^{5})s+(8t^{2}+12t^{3}+8t^{4})s^{2}+(14t^{2}+30t^{3}+14t^{4})s^{3}+(4t^{2}+24t^{3}+4t^{4})s^{4}$}\tabularnewline
\hline 
{\small{}$6$} & {\small{}$(t+t^{6})s+(10t^{2}+20t^{3}+20t^{4}+10t^{5})s^{2}+(28t^{2}+90t^{3}+90t^{4}+28t^{5})s^{3}+(14t^{2}+136t^{3}+136t^{4}+14t^{5})s^{4}+(5t^{2}+56t^{3}+56t^{4}+5t^{5})s^{5}$}\tabularnewline
\hline 
\end{tabular}{\small\par}
\par\end{raggedright}
\caption{\label{tb-brides}Joint distribution of $\protect\br$ and $\protect\ides$
over $\mathfrak{S}_{n}$}
\end{table}

The first several polynomials $P_{n}^{(\br,\ipk)}(s,t)$ and $P_{n}^{(\br,\ides)}(s,t)$
are displayed in Tables \ref{tb-bripk}\textendash \ref{tb-brides}.
Notice that the coefficients of $s^{k}t$ in $P_{n}^{(\br,\ipk)}(t)$
are twice the coefficients of $s^{k}t$ in $P_{n-1}^{(\udr,\ipk)}(t)$;
we shall give a simple proof of this fact.

\begin{prop}
\label{p-bripk0}For any $n\geq2$ and $k\geq0$, the number of permutations
in $\mathfrak{S}_{n}$ with $\br(\pi)=k$ and $\ipk(\pi)=0$ is equal
to $2{n-2 \choose k-1}$, twice the number of permutations in $\mathfrak{S}_{n-1}$
with $\udr(\pi)=k$ and $\ipk(\pi)=0$.
\end{prop}

\begin{proof}
In light of Proposition \ref{p-udripk0} and its proof, it suffices
to construct a 1-to-2 map from compositions $L\vDash n-1$ with $k$
parts to compositions $K\vDash n$ satisfying $\br(K)=k$.
We claim that such a map is given by sending $L=(L_{1},L_{2},\dots,L_{k})$
to the descent compositions corresponding to the up-down compositions
$J=(L_{1}+1,L_{2},\dots,L_{k})$ and $J^{\prime}=(1,L_{1},L_{2},\dots,L_{k})$.
It is easy to see that $\br(\pi)=k$ if and only if $\udComp(\pi)$
has $k$ parts with initial part greater than 1 or if $\udComp(\pi)$
has $k+1$ parts with initial part 1, so this map is well-defined.
For example, let $L=(4,1,1,2)$. Then permutations with up-down compositions
$J=(5,1,1,2)$ and $J^{\prime}=(1,4,1,1,2)$ have descent compositions
$K=(5,2,1,1)$ and $K^{\prime}=(1,1,1,1,2,3)$, respectively, so our
map sends $L$ to $K$ and $K^{\prime}$. The reverse procedure is
given by first taking the up-down composition corresponding to the
descent composition input, and then removing the first part if it
is equal to 1 and subtracting 1 from the first part otherwise.
\end{proof}

In Table \ref{tb-brides}, we first collect powers of $s$ in displaying
the $P_{n}^{(\br,\ides)}(s,t)$ to showcase the symmetry present in
the coefficients of $s^{k}$. This symmetry can be explained in the
same way as for the polynomials $P_{n}^{(\pk,\ides)}(s,t)$: upon
taking reverses, we have that $\ides(\pi^{r})=n-1-\ides(\pi)$ but
$\br(\pi^{r})=\br(\pi)$. Note that the coefficients of $s^{k}$ for
$k\geq2$ also seem to be unimodal, but in general they are not $\gamma$-positive.

\section{\label{s-maj}Major index}

\subsection{A rederivation of the Garsia\textendash Gessel formula}

We now return full circle by showing how our approach can
be used to rederive Garsia and Gessel's formula (\ref{e-gg}) for
the polynomials $A_{n}(s,t,q,r)=\sum_{\pi\in\mathfrak{S}_{n}}s^{\des(\pi)}t^{\ides(\pi)}q^{\maj(\pi)}r^{\imaj(\pi)}$.

\begin{proof}[Proof of the Garsia--Gessel formula]
We seek to compute the the scalar product
\[
\left\langle \sum_{i=0}^{\infty}s^{i}\prod_{k=0}^{i}H(q^{k}x),\sum_{j=0}^{\infty}t^{j}\prod_{l=0}^{j}H(r^{l})\right\rangle 
\]
in two different ways. First, from Lemma \ref{l-ribexp} (a) we have{\allowdisplaybreaks
\begin{alignat}{1}
 & \left\langle \sum_{i=0}^{\infty}s^{i}\prod_{k=0}^{i}H(q^{k}x),\sum_{j=0}^{\infty}t^{j}\prod_{l=0}^{j}H(r^{l})\right\rangle \nonumber \\
 & \qquad\qquad\qquad=\sum_{m,n=0}^{\infty}\frac{\sum_{L\vDash m,M\vDash n}s^{\des(L)}t^{\des(M)}q^{\maj(L)}r^{\maj(M)}\left\langle r_{L},r_{M}\right\rangle }{(1-s)(1-qs)\cdots(1-q^{m}s)(1-t)(1-rt)\cdots(1-r^{n}t)}x^{m}\nonumber \\
 & \qquad\qquad\qquad=\sum_{n=0}^{\infty}\frac{\sum_{\pi\in\mathfrak{S}_{n}}s^{\des(\pi)}t^{\ides(\pi)}q^{\maj(\pi)}r^{\imaj(\pi)}}{(1-s)(1-qs)\cdots(1-q^{n}s)(1-t)(1-rt)\cdots(1-r^{n}t)}x^{n}\nonumber \\
 & \qquad\qquad\qquad=\sum_{n=0}^{\infty}\frac{A_{n}(s,t,q,r)}{(1-s)(1-qs)\cdots(1-q^{n}s)(1-t)(1-rt)\cdots(1-r^{n}t)}x^{n}\label{e-desmaj1}
\end{alignat}
}where, as usual, we apply Foulkes's theorem (Theorem \ref{t-foulkes})
to calculate the scalar product $\left\langle r_{L},r_{M}\right\rangle $.
Second, observe that 
\begin{align*}
\left\langle \sum_{i=0}^{\infty}s^{i}\prod_{k=0}^{i}H(q^{k}x),\sum_{j=0}^{\infty}t^{j}\prod_{l=0}^{j}H(r^{l})\right\rangle  & =\sum_{i,j=0}^{\infty}s^{i}t^{i}\left\langle \prod_{k=0}^{i}H(q^{k}x),\prod_{l=0}^{j}H(r^{l})\right\rangle \\
 & =\sum_{i,j=0}^{\infty}s^{i}t^{i}\prod_{k=0}^{i}\prod_{l=0}^{j}\left\langle H(q^{k}x),H(r^{l})\right\rangle 
\end{align*}
by Lemma \ref{l-desmajhom}. To evaluate $\left\langle H(q^{k}x),H(r^{l})\right\rangle $,
let us recall that $h_{n}=\sum_{\lambda\vdash n}m_{\lambda}$. Then
{\allowdisplaybreaks
\begin{align*}
\left\langle H(q^{k}x),H(r^{l})\right\rangle  & =\left\langle \sum_{i=0}^{\infty}h_{i}q^{ki}x^{i},\sum_{j=0}^{\infty}h_{j}r^{lj}\right\rangle \\
 & =\sum_{i=0}^{\infty}\sum_{j=0}^{\infty}\sum_{\lambda\vdash i}\left\langle m_{\lambda},h_{j}\right\rangle q^{ki}x^{i}r^{lj}\\
 & =\sum_{j=0}^{\infty}q^{kj}x^{j}r^{lj}\\
 & =\frac{1}{1-xq^{k}r^{l}},
\end{align*}
}whence 
\begin{equation}
\left\langle \sum_{i=0}^{\infty}s^{i}\prod_{k=0}^{i}H(q^{k}x),\sum_{j=0}^{\infty}t^{j}\prod_{l=0}^{j}H(r^{l})\right\rangle =\sum_{i,j=0}^{\infty}s^{i}t^{i}\prod_{k=0}^{i}\prod_{l=0}^{j}\frac{1}{1-xq^{k}r^{l}}.\label{e-desmaj2}
\end{equation}
Combining (\ref{e-desmaj1}) and (\ref{e-desmaj2}) yields Garsia
and Gessel's formula (\ref{e-gg}).
\end{proof}

\subsection{Major index and other statistics}

Lemma \ref{l-ribexp} (a) also enables us to derive formulas for mixed
two-sided distributions involving the major index. Define
\begin{align*}
P_{n}^{(\maj,\ipk,\des,\ides)}(q,y,s,t) & \coloneqq\sum_{\pi\in\mathfrak{S}_{n}}q^{\maj(\pi)}y^{\ipk(\pi)+1}s^{\des(\pi)}t^{\ides(\pi)+1},\\
P_{n}^{(\maj,\ilpk,\des,\ides)}(q,y,s,t) & \coloneqq\sum_{\pi\in\mathfrak{S}_{n}}q^{\maj(\pi)}y^{\ilpk(\pi)}s^{\des(\pi)}t^{\ides(\pi)},\quad\text{and}\\
P_{n}^{(\maj,\iudr,\des)}(q,s,t) & \coloneqq\sum_{\pi\in\mathfrak{S}_{n}}q^{\maj(\pi)}s^{\iudr(\pi)}t^{\des(\pi)}.
\end{align*}
We leave the proofs of the following formulas to the interested reader.
\begin{thm}
\label{t-majmixed} We have \leqnomode
\begin{multline*}
\tag{{a}}\qquad\frac{1}{(1-s)(1-t)}+\frac{1}{1+y}\sum_{n=1}^{\infty}\left(\frac{1+yt}{1-t}\right)^{n+1}\frac{P_{n}^{(\maj,\ipk,\des,\ides)}\left(q,\frac{(1+y)^{2}t}{(y+t)(1+yt)},s,\frac{y+t}{1+yt}\right)}{(1-s)(1-qs)\cdots(1-q^{n}s)}x^{n}\\
=\sum_{i,j=0}^{\infty}s^{i}t^{j}\prod_{k=0}^{i}\left(\frac{1+q^{k}yx}{1-q^{k}x}\right)^{j},
\end{multline*}
\begin{multline*}
\tag{{b}}\qquad\frac{1}{(1-s)(1-t)}+\sum_{n=1}^{\infty}\frac{(1+t)^{n}}{(1-t)^{n+1}}\frac{P_{n}^{(\maj,\ilpk,\des,\ides)}\left(q,\frac{(1+y)^{2}t}{(y+t)(1+yt)},s,\frac{y+t}{1+yt}\right)}{(1-s)(1-qs)\cdots(1-q^{n}s)}x^{n}\\
=\sum_{i,j=0}^{\infty}s^{i}t^{j}\prod_{k=0}^{i}\frac{(1+q^{k}yx)^{j}}{(1-q^{k}x)^{j+1}},
\end{multline*}
and
\begin{multline*}
\tag{{c}}\qquad\frac{1}{(1-s)(1-t)}+\frac{1}{2(1-s)^{2}}\sum_{n=1}^{\infty}\frac{(1+s^{2})^{n}}{(1-s^{2})^{n-1}}\frac{P_{n}^{(\maj,\iudr,\des)}\left(q,\frac{2s}{1+s^{2}},t\right)}{(1-t)(1-qt)\cdots(1-q^{n}t)}\\
=\sum_{i,j=0}^{\infty}s^{2i}t^{j}\left(1+s\prod_{k=0}^{j}\frac{1}{1-q^{k}x}\right)\prod_{k=0}^{j}\left(\frac{1+q^{k}x}{1-q^{k}x}\right)^{i}.
\end{multline*}
\end{thm}

One may obtain formulas for $(\maj,\ipk)$,
$(\maj,\ilpk)$, and $(\maj,\iudr)$ by specializing Theorem \ref{t-majmixed}
appropriately. For example, setting $y=1$ in Theorem \ref{t-majmixed}
(a), multiplying both sides by $1-s$, and then taking the limit of
both sides as $s\rightarrow1$ yields
\[
\frac{1}{1-t}+\frac{1}{2}\sum_{n=1}^{\infty}\left(\frac{1+t}{1-t}\right)^{n+1}\frac{P_{n}^{(\maj,\ipk)}\left(q,\frac{4t}{(1+t)^{2}}\right)}{(1-q)(1-q^{2})\cdots(1-q^{n})}x^{n}=\sum_{j=0}^{\infty}t^{j}\prod_{k=0}^{\infty}\left(\frac{1+q^{k}x}{1-q^{k}x}\right)^{j}
\]
where $P_{n}^{(\maj,\ipk)}(q,t)\coloneqq P_{n}^{(\maj,\ipk,\des,\ides)}(q,t,1,1)=\sum_{\pi\in\mathfrak{S}_{n}}q^{\maj(\pi)}t^{\ipk(\pi)+1}$.

\section{\label{s-conj}Conjectures}

We conclude with a discussion of some conjectures concerning some
of the permutation statistic distributions studied in this paper.

\subsection{Real-rootedness}

A univariate polynomial is called \textit{real-rooted} if it has only
real roots. We conjecture that the distributions of $\ipk$ and $\ilpk$
over permutations in $\mathfrak{S}_{n}$ with any fixed value of $\pk$,
$\lpk$, and $\udr$\textemdash as well as that of $\ipk$ upon fixing
$\br$\textemdash are all encoded by real-rooted polynomials. This
conjecture has been empirically verified for all $n\leq50$; our formulas
from Sections \ref{s-pkdes}\textendash \ref{s-maj} played a crucial
role in formulating and gathering supporting evidence for this conjecture
as they have allowed us to efficiently compute the polynomials in
question.
\begin{conjecture}
\label{c-rr}The following polynomials are real-rooted for all $n\geq1$:
\begin{enumerate}
\item [\normalfont{(a)}] ${\displaystyle [s^{k+1}]\,P_{n}^{(\pk,\ipk)}(s,t)=\sum_{\substack{\pi\in\mathfrak{S}_{n}\\
\pk(\pi)=k
}
}}t^{\ipk(\pi)+1}$ for  $0\leq k\leq\left\lfloor (n-1)/2\right\rfloor $,
\item [\normalfont{(b)}] ${\displaystyle [s^{k}]\,P_{n}^{(\lpk,\ilpk)}(s,t)=\sum_{\substack{\pi\in\mathfrak{S}_{n}\\
\lpk(\pi)=k
}
}}t^{\ilpk(\pi)}$ for  $0\leq k\leq\left\lfloor n/2\right\rfloor $,
\item [\normalfont{(c)}] ${\displaystyle [s^{k}]\,P_{n}^{(\pk,\ilpk)}(t,s)=\sum_{\substack{\pi\in\mathfrak{S}_{n}\\
\lpk(\pi)=k
}
}}t^{\ipk(\pi)+1}$ for  $0\leq k\leq\left\lfloor n/2\right\rfloor $,
\item [\normalfont{(d)}] ${\displaystyle [s^{k+1}]\,P_{n}^{(\pk,\ilpk)}(s,t)=\sum_{\substack{\pi\in\mathfrak{S}_{n}\\
\pk(\pi)=k
}
}}t^{\ilpk(\pi)}$ for  $0\leq k\leq\left\lfloor (n-1)/2\right\rfloor $,
\item [\normalfont{(e)}] ${\displaystyle [s^{k}]\,P_{n}^{(\udr,\ipk)}(s,t)=\sum_{\substack{\pi\in\mathfrak{S}_{n}\\
\udr(\pi)=k
}
}}t^{\ipk(\pi)+1}$ for  $1\leq k\leq n$,
\item [\normalfont{(f)}] ${\displaystyle [s^{k}]\,P_{n}^{(\udr,\ilpk)}(s,t)=\sum_{\substack{\pi\in\mathfrak{S}_{n}\\
\udr(\pi)=k
}
}}t^{\ilpk(\pi)}$ for  $1\leq k\leq n$, and
\item [\normalfont{(g)}] ${\displaystyle [s^{k}]\,P_{n}^{(\br,\ipk)}(s,t)=\sum_{\substack{\pi\in\mathfrak{S}_{n}\\
\br(\pi)=k
}
}}t^{\ipk(\pi)+1}$ for  $1\leq k\leq n$.
\end{enumerate}
\end{conjecture}

Conjecture \ref{c-rr} would imply that all of these polynomials are
unimodal and log-concave, and can be used to show that the distributions 
of $\ipk$ and $\ilpk$ over permutations in $\mathfrak{S}_{n}$ with a fixed value
$k$ for the relevant statistics each converge to a normal distribution
as $n\rightarrow\infty$. 

It is worth noting that the peak and left peak polynomials 
\[
P_{n}^{\pk}(t)=\sum_{\pi\in\mathfrak{S}_{n}}t^{\pk(\pi)+1}=\sum_{\pi\in\mathfrak{S}_{n}}t^{\ipk(\pi)+1}\quad\text{and}\quad P_{n}^{\lpk}(t)=\sum_{\pi\in\mathfrak{S}_{n}}t^{\lpk(\pi)}=\sum_{\pi\in\mathfrak{S}_{n}}t^{\ilpk(\pi)}
\]
are known to be real-rooted (see, e.g., \cite{Petersen2007,Stembridge1997}),
so Conjecture \ref{c-rr} would also have the consequence that both
of these ``real-rooted distributions'' can be partitioned into real-rooted
distributions based on the value of another statistic.

\subsection{Gamma-positivity}

Real-rootedness provides a powerful method of proving unimodality
results in combinatorics, but when the polynomials in question are
known to be symmetric,\footnote{The use of the term ``symmetric polynomial'' in this context is different
from that in symmetric function theory. Here, a symmetric polynomial
refers to a univariate polynomial whose coefficients form a symmetric
sequence.} an alternate avenue to unimodality is $\gamma$-positivity. Any symmetric
polynomial $f(t)\in\mathbb{R}[t]$ with center of symmetry $n/2$
can be written uniquely as a linear combination of the polynomials
$\{t^{j}(1+t)^{n-2j}\}_{0\leq j\leq\left\lfloor n/2\right\rfloor }$\textemdash referred
to as the \textit{gamma basis}\textemdash and $f(t)$ is called \textit{$\gamma$-positive}
if its coefficients in the gamma basis are non-negative. The $\gamma$-positivity
of a polynomial directly implies its unimodality, and $\gamma$-positivity
has appeared in many contexts within combinatorics and geometry; see
\cite{Athanasiadis2017} for a detailed survey.

The prototypical example of a family of $\gamma$-positive polynomials
are the Eulerian polynomials $A_{n}(t)$, as established by Foata
and Sch\"{u}tzenberger \cite{Foata1970}, and there is also a sizable
literature on the $\gamma$-positivity of ``Eulerian distributions''
(polynomials encoding the distribution of the descent number $\des$)
over various restricted subsets of $\mathfrak{S}_{n}$. For example,
the Eulerian distribution over linear extensions of sign-graded posets
\cite{Braenden2004/06}, $r$-stack-sortable permutations \cite{Braenden2008},
separable permutations \cite{Fu2018}, and involutions \cite{Wang2019}
are all known to be $\gamma$-positive. The two-sided Eulerian distribution 
$(\des,\ides)$ is also known to satisfy a refined $\gamma$-positivity 
property which was conjectured by Gessel and later proved by Lin \cite{Lin2016a}.

Define
\[
\hat{A}_{n,k}(t)\coloneqq\sum_{\substack{\pi\in\mathfrak{S}_{n}\\
\pk(\pi)=k
}
}t^{\ides(\pi)+1}
\]
to be the polynomial encoding the distribution of $\ides$ over permutations
in $\mathfrak{S}_{n}$ with $k$ peaks, or equivalently, the distribution
of $\des$ (i.e., the Eulerian distribution) over permutations in
$\mathfrak{S}_{n}$ whose inverses have $k$ peaks. We conjecture that
these polynomials are $\gamma$-positive as well.
\begin{conjecture}
\label{cj-gp}For all $n\geq1$ and $0\leq k\leq\left\lfloor (n-1)/2\right\rfloor $,
the polynomials $\hat{A}_{n,k}(t)$ are $\gamma$-positive with center
of symmetry $(n+1)/2$.
\end{conjecture}

In fact, we shall give a stronger conjecture which refines by ``pinnacle
sets''. Given a permutation $\pi\in\mathfrak{S}_{n}$, the \textit{pinnacle
set} $\Pin(\pi)$ of $\pi$ is defined by 
\[
\Pin(\pi)\coloneqq\{\,\pi(k):\pi(k-1)<\pi(k)>\pi(k+1)\text{ and }2\leq k\leq n-1\,\}.
\]
In other words, $\Pin(\pi)$ contains all of the values $\pi(k)$
at which $k$ is a peak of $\pi$. Given $n\geq1$ and $S\subseteq[n]$,
define the polynomial $\hat{A}_{n,S}^{\ides}(t)$ by
\[
\hat{A}_{n,S}^{\ides}(t)\coloneqq\sum_{\substack{\pi\in\mathfrak{S}_{n}\\
\Pin(\pi)=S
}
}t^{\ides(\pi)+1};
\]
this gives the distribution of the inverse descent number over permutations
in $\mathfrak{S}_{n}$ with a fixed pinnacle set $S$, or equivalently,
the Eulerian distribution over permutations in $\mathfrak{S}_{n}$
with ``inverse pinnacle set'' $S$.
\begin{conjecture}
\label{cj-gppin}For all $n\geq1$ and $S\subseteq[n]$, the polynomials
$\hat{A}_{n,S}^{\ides}(t)$ are $\gamma$-positive with center of
symmetry $(n+1)/2$.
\end{conjecture}

Note that
\[
\hat{A}_{n,k}^{\ides}(t)=\sum_{\substack{S\subseteq[n]\\
\left|S\right|=k
}
}\hat{A}_{n,S}^{\ides}(t).
\]
Because the sum of $\gamma$-positive polynomials with the same center
of symmetry is again $\gamma$-positive, a positive resolution to
Conjecture \ref{cj-gppin} would imply Conjecture \ref{cj-gp}. However,
we do not have nearly as much empirical evidence to support Conjecture
\ref{cj-gppin}. Since we do not have a formula for the polynomials
$\hat{A}_{n,k}^{\ides}(t)$, we were only able to verify Conjecture
\ref{cj-gppin} up to $n=10$, whereas Conjecture \ref{cj-gp} has
been verified for all $n\leq80$ with the assistance of Theorem \ref{t-pkides}
(a).

\bigskip{}

\noindent \textbf{Acknowledgements.} We thank Kyle
Petersen for insightful discussions concerning the work in this paper,
and in particular for his suggestion to look at refining Conjecture
\ref{cj-gp} by pinnacle sets. We also thank two anonymous referees for their
careful reading of an earlier version of this paper and providing
thoughtful comments.

\bibliographystyle{plain}
\addcontentsline{toc}{section}{\refname}\bibliography{bibliography}

\end{document}